\documentclass[12pt]{amsart}

\usepackage[a4paper]{geometry}
\usepackage[centertags]{amsmath}
\usepackage{amsfonts}
\usepackage{amssymb}
\usepackage{amsmath}
\usepackage{amsthm}
\usepackage{graphicx}
\usepackage{subfigure}
\usepackage{times}
\usepackage{color}
\usepackage{amsaddr}
\usepackage[mathcal]{euscript}

\numberwithin{equation}{section}%%

{\bf}{\it}
\newtheorem*{theorem}{Theorem}
\newtheorem{proposition}{Proposition}[section]

\theoremstyle{definition}

\newtheorem{corollary}[proposition]{Corollary}
\newtheorem{lemma}[proposition]{Lemma}

\newtheorem{definition}{Definition}

\theoremstyle{remark}
\newtheorem{remark}{Remark}

\newcommand{\be}{\begin{equation}}
\newcommand{\ee}{\end{equation}}
\newcommand{\N}{\mathbb{N}}
\newcommand{\Z}{\mathbb{Z}}

\newcommand{\mo}{\mathcal{O}}
\newcommand{\bm}{\bar{\mathcal{M}}}

\begin{document}

\author[Stefano Luzzatto]{Stefano Luzzatto}
\address[Stefano Luzzatto]{International Center for Theoretical
Physics (ICTP)\\
Strada Costiera 11, Trieste, Italy}
\email[Stefano Luzzatto]{luzzatto@ictp.it}

\author[Marks Ruziboev]{Marks Ruziboev }
\address[Marks Ruziboev]{International School for Advanced Studies (SISSA)\\
Via Bonomea 265, Trieste, Italy
 \\International Center for Theoretical
Physics (ICTP)\\
Strada Costiera 11, Trieste, Italy}
\email[Marks Ruziboev]{mruziboe@sissa.it}

\title[Young towers for product systems]{Young towers for product systems}

\date{28 April 2015}

\subjclass[2000]{37A05, 37A25}

\keywords{Decay of correlations, product systems, Young Towers.}

\thanks{We would like to thank Viviane Baladi and Carlangelo Liverani
for several useful discussions and their help on some parts of the paper, and a referee for their comments which helped us significantly improve the presentation of the paper.}

\begin{abstract}
We show that the direct product of maps with Young towers admits a Young tower whose return times decay at a rate which is bounded above by the slowest of the rates of decay of the return times of the component maps.
An application of this result, together with other results in the literature, yields various statistical properties for the direct product of various classes of systems, including Lorenz-like maps, multimodal
maps, piecewise $C^2$ interval maps with critical points and
singularities, H\'enon maps and partially hyperbolic systems.
\end{abstract}

\maketitle

\section{Introduction and statement of results}

\subsection{Product systems and Young towers}
Let \( f_i:M_i\to M_i \),  $i=1, ..., \ell,$ be a family of maps defined on a family of Riemannian manifolds. Define the product map \(  f=f_1\times ...
\times f_\ell \) on
\(
 {M}~{=M_1\times ... \times M_\ell}
 \) by
\begin{equation}\label{eq:prod}
f(x_1,...,x_\ell)=(f_1(x_1),..,f_\ell(x_\ell)).
\end{equation}
Dynamical properties of the product system \( f \) can be partially
but not completely deduced from the dynamical properties of its
components, as the product system may exhibit significantly richer
dynamics. Even in  the simplest setting of the product of two
identical maps \( g\times g: M \to M \), if \( p, q \) are periodic
points for \( g \), then the set-theoretic product of the periodic
\emph{orbits} \( \mathcal O^{+}(p)=\{p_{1},...,p_{m}\} \) and
\(\mathcal O^{+}(q)=\{q_{1},...,q_{n}\} \) may consist of several
periodic orbits\footnote{ In fact, if  gcd\(\{m, n\}=1\), then $(p,
q)$ is periodic point of period $mn.$ On the other hand, if gcd$\{m
,n\}=k>1,$ then $(p, q)$ is periodic point with period $mn/k,$ but
$\mo^+(p)\times\mo^+(q)$ is a union of $k$ periodic orbits. As a
simple example, consider the case $m=2$ and $n=4.$ In this case, the
product of the orbits splits into two orbits for the product map:
$\mo^+(p)\times\mo^+(q)=\mo^+(p_1\times q_1)\cup\mo^+(p_2\times
q_1).$ }, which are in some sense \emph{new} periodic orbits which
do not exist in either of the original systems. Similarly, from an
ergodic-theoretic point of view, the space \( \mathcal M_{g\times g}
\) of invariant probability measures for the product map contains
many invariant measures which are not products of invariant measures
for \( g \) such as  measures supported on the ``new'' periodic
orbits mentioned above
%\footnote{
%As above let $p$ and $q$ be periodic
%points of period $m$ and $n$ respectively for $f$ and $g.$ Then the
%Dirac measures $\delta_{\mo^+(p)}=\frac 1 m\sum\delta_{p_i}$ and
%$\delta_{\mo^+(q)}=\frac 1 n \sum\delta_{q_j}$ preserved by $f$ and
%$g$ respectively. In the case gcd$\{m ,n\}=k>1,$ the Dirac measure
%defined as $\delta_{\mo^+(p\times q)}
%=\frac{k}{mn}\sum\delta_{p_i\times q_j},$ $p_i\times
%q_j\in\mo^+(p\times q)$ is preserved by $f\times g,$ but it is not a
%product of measures $\delta_{\mo^+(p)}$ and $\delta_{\mo^+(q)}.$
%}
and measures supported on the ``diagonal'' \(
\{(x,x): x\in M\}\subset M \times M  \) which  is invariant for the  product.\footnote{Thanks to M. Blank for this interesting
observation.}

One important tool that has emerged in the last decade or so for the
study of dynamical and statistical properties of various kinds of
dynamical systems, is that of \emph{Young towers}, after the
pioneering work of Young \cite{Y}. The direct product of Young
towers is not in general itself a Young tower, and so it is not
immediately obvious that the direct product of systems which admit a
Young tower also admits a Young tower. The main result of this paper
is to show that, nevertheless, a Young tower can be constructed  for
the product system \eqref{eq:prod} if each component \(  f_{i}  \)
admits a Young tower, and that  we can obtain some estimates for the
decay of the return times of the tower for the product in terms of
the rates of decay for the individual towers. We will also discuss
various applications of this result to products of systems for which
Young towers are known to exist.

\subsection{Statement of results}
We postpone the precise formal definition of a Young tower to
Section \ref{sec:YT} since it is somewhat technical and our main
results can be stated without any reference to the details. The
definition is slightly  different depending on whether \( f \) is
invertible or not  but both cases we have a reference probability
measure \( m \) on a "base" \( \Delta_{0}\subset M \), and a return
time function \( R: \Delta_{0}\to \mathbb N  \). We suppose that each
of our maps \( f_i: M_i\to M_i \), \( i=1,...,\ell \) admit Young
towers with reference measures \( m^{(i)}_{0}\) defined on bases \( \Delta_{0}^{(i)} \)
with return time functions \( R^{(i)}: \Delta^{(i)}_{0}\to \mathbb N \). Let \(
f: M \to M \) denote the corresponding product system
\eqref{eq:prod}, and $\bar{m}_{0}=m^{(1)}_{0}\times \cdots \times m^{(\ell)}_{0}$ denote the
product measure on \( \bar\Delta_{0}=\Delta^{(1)}_{0} \times \cdots \times
\Delta^{(\ell)}_{0}\). Finally, let
\begin{equation}\label{M}
\mathcal{M}_n=\max_{i=1,.., \ell}m^{(i)}_{0}\{R^{(i)}>n\}.
\end{equation}
We  always assume that either all of the maps are invertible
or all are noninvertible so that the product system is of the same type.

\begin{theorem}\label{suytower}
The product system $f$ admits a tower
with reference measure \( \bar m_{0} \) on  $\bar\Delta_{0}$ and  return time \( T: \Delta_{0}\to \mathbb N \) satisfying the following bounds.

\medskip
\noindent
\textbf{Exponential Decay:}
if
$\mathcal M_n=\mathcal{O}(e^{-\tau n})$ for some   $\tau>0$, then
\[
\bar m_{0}\{T>n\}=\mathcal{O}(e^{-\tau'n}).
\]
for some \( \tau'>0 \).

\medskip
\noindent
\textbf{Stretched Exponential Decay}
 if $\mathcal M_n=\mathcal{O}(e^{-\tau n^\theta})$ for some  $\tau,\theta>0$, then
\[
\bar m_{0}\{T>n\}=\mathcal{O}(e^{-\tau'n^{\theta'}});
\]
for all \( 0<\theta'<\theta \)  and some  \( \tau'=\tau'(\theta')>0 \).

 \medskip
\noindent
\textbf{Polynomial Decay}
 if $\mathcal M_n=\mathcal{O}(n^{-\alpha})$ for some $\alpha>\ell$, then
 \[
\bar m_{0}\{T>n\}=\mathcal{O}(n^{\ell-\alpha}).
 \]
\end{theorem}

The idea of constructing a Young tower for a product of Young towers is sketched, in the non-invertible setting with exponential or polynomial return times,
without a fully developed proof, in the   PhD thesis  of Vincent Lynch \cite{LyTh}.
Our construction
and estimates lead to a "loss" of one exponent for each component of the product system in the case of polynomial rates of decay, and also  do not allow us to obtain results if the decay is slower than polynomial, such as in the
  interesting examples of  Holland  \cite{Hol} which exhibit decay at  rates of the form $(\log\circ...\circ\log n)^{-1}$.
It is not clear to us if this is just a technical issue  or if there might be some deeper reasons.

\subsection{Applications}
Young towers  have been shown to imply a variety of
statistical properties such as decay of correlations, invariance
principles, limit theorems which in some cases can also be
quantified in terms of the rate of decay of the
tail of the return times associated to the tower \cite{
AlvPin, Che99, Chern, Gou, Ly,  MelbNic, MelbNic05, Y, LSY}.
An immediate consequence is that the dynamical systems which are direct products of systems which admit Young towers satisfy the statistical properties corresponding to the tail estimates of the product as given in our Theorem.
Some examples of these systems include the following.
 \emph{Lorenz-like interval maps} which are uniformly expanding and have a
single singularity with dense preimages satisfying \( |f'(x)|\approx
|x|^{-\beta}  \) for some \( \beta\in (1/2, 1)  \)
admit a Young Tower with exponential tail, see \cite{DO}
for the precise technical conditions;
 \emph{Multimodal maps} for which  the decay rate of the return times was obtained in terms of the
growth rate of the derivative along the critical orbits \cite{BLS};
 \emph{Maps with critical points and singularities} in one and also higher dimensions \cite{ALP, DHL, AlvSch13};
 \emph{Planar periodic Lorentz gas} was introduced by Sinai \cite{S2}
and  admits Young Towers with exponential tails \cite{LSY,Chern}; \emph{H{\'e}non maps}, for certain choices of parameters \(  (a,b)  \), the maps $H_{a, b}:\mathbb R^2\to \mathbb R^2$ given by
 $H_{a, b}(x, y)=(1-ax^2+y, bx)$  admit a Young Tower with exponential tail \cite{BenCar, BenYou95}; \emph{Partially hyperbolic systems} under certain additional conditions \cite{AlvLi, AlvPin}.

We emphasize that the construction of Young towers, and especially the estimation of the decay of the return times, is in general highly non-trivial and relies on the specific geometric and dynamical properties of the system under consideration. The geometry of the direct product of any of the systems mentioned above is extremely complicated and it is doubtful that a tower construction could be achieved without taking advantage of the information that each component admits a tower with certain decay rates and applying our theorem.

In certain cases, combining our result with existing literature, it is possible to deduce statistical properties of the product system directly from the statistical properties of the component systems without assuming a priori that these admit Young towers. Indeed,  in certain settings, such as that of non-uniformly expanding systems,
statistical properties such as decay of correlations or large deviations imply the existence of Young tower with corresponding tail estimates  see e.g. \cite[Theorem 4.2]{AFLV}.
 Thus, taking the direct product of any finite number of such systems we can apply our result and those of \cite{MelbNic} to conclude that Large Deviations for the component systems implies Large Deviations for the product system.
The same argument also works for  Decay of Correlations but in this case, as Carlangelo Liverani  pointed out to us, it is possible to give an elementary and self-contained direct proof. Since it does not appear to exist in the literature, we include it in the Appendix of this paper.

\subsection{Overview of the paper}
In Section \ref{sec:YT} we give the precise formal definition of a Young tower  in the  non-invertible setting. Most of the paper will then be dedicated to proving our main result in this setting. In Section \ref{pofsuytower} we will give the formal definition  in the invertible case  and show how the  construction easily extends to this setting.
In Section \ref{tfprod} we give the full combinatorial/topological construction of the tower for the product system.  In Section \ref{asymp} we give some basic estimates concerning the asymptotics of the tail of the newly constructed tower for the product, and in Section \ref{pfnew} we apply these to get specific estimates in the exponential, stretched exponential and polynomial case. In Section \ref{final} combining the results from previous sections, we finalize proof of the Theorem in the non-invertible case. Finally, in the appendix we prove some technical estimates and  explain the estimate due to Liverani on the decay of correlations for product systems.

\section{Young Towers}\label{sec:YT}

\subsection{Gibbs-Markov-Young Towers}
\label{sec:GMY}
We start  with the formal definition of Young Tower for non-invertible
maps. To distinguish this case from the tower for invertible maps, this structure
sometimes is referred to as a Gibbs-Markov-Young (GMY) structure or
GMY-tower. Let $f:M\to M$ be a $C^{1+\alpha},$ $\alpha\in(0, 1)$ local diffeomorphism
(outside some critical/singular set) of a Riemannian manifold $M$ on which we have a Riemannian volume  which we will refer to as Lebesgue measure.
To define a  GMY-structure for \(  f  \) we  start with a measurable set  $\Delta_0$ with finite positive Lebesgue measure, we let \(  m_{0}  \) denote the restriction of Lebesgue measure to \(  \Delta_{0}  \),
a mod 0 partition $\mathcal P=\{\Delta_{0, i}\}$ of \(  \Delta_{0}  \), and a return
time function $R:\Delta_0\to \mathbb N$ that is constant on the
partition elements, i.e. $R|\Delta_{0, i}=R_i$  such that the induced map \(  f^{R}:\Delta_{0}\to\Delta_{0}  \) is well defined. Then, for any two point \(  x, y\in \Delta_{0}  \) we can define   the \textit{separation time } $s(x, y)$ as the  smallest  $k\geq 0$ such that $(f^R)^k(x) $ and $(f^R)^k(y)$ lie in different
partition elements and assume that there exists \(  \beta\in (0,1)  \)  and \(  D>0  \)
such that the following conditions are satisfied.

\begin{itemize}\label{tower}
\item[(G1)] \textbf{Markov}: for any $\Delta_{0,i}\in\mathcal{P}$
 the map $f^{R_i}:\Delta_{0,i} \to\Delta_0$ is a bijection.

\item[(G2)] \textbf{Uniform expansion}:  $\|(Df^R)^{-1}(x)\|\leq \beta$  for \(  m_{0}  \) a.e. \(  x  \).

\item[(G3)] \textbf{Bounded distortion}: for a.e.  pair  $x,y\in\Delta_0$ with $s(x, y)<\infty$
we have
\begin{equation}\label{bdd}
\left|\frac{\text{det}Df^R(x)}{\text{det}Df^R(y)}-1\right|\leq
D\beta^{s(f^Rx, f^Ry)}.
\end{equation}
\item[(G4)] \textbf{Integrability}: $\int R\,dm_0<\infty$.
\item[(G5)] \textbf{Aperiodicity}: gcd$\{R_i\}=1.$
\end{itemize}

Young showed that the first three assumptions (G1)-(G3) imply  that the map \(  f  \) admits an \(  f  \)-invariant  measure \(  \mu  \) which is absolutely continuous with respect to Lebesgue. Condition (G4) implies that this measure is finite, and therefore can be taken to be a probability measure, and condition (G5) implies that it is mixing.
We note that the integrability and the aperiodicity assumption are not always included in the definition of a Young Tower. We include them here because we actually require the existence of a mixing probability measure for each component for our argument to work in the proof of the Theorem and also because we are interested in applications to the problem of the rate of decay of correlations for product systems and this requires the measures involved to be mixing probability measures.

We note also  that the return time \(  R  \) is not generally a \emph{first return} time of point to \(  \Delta_{0}  \). It is therefore often useful to work with an ``extension'' of \(  f  \) in which the returns to \(  \Delta_{0}  \) are first return times. This extension is precisely what we refer to as a   \emph{GMY-tower}. The formal construction of this extension proceeds as follows.  We let
\be\label{tower1}
 \Delta=\{(z,n)\in\Delta_0\times\mathbb{Z}^+_0|\,\, R(z)>n\},
 \ee where
$\mathbb{Z}_0^+$ denotes the set of all nonnegative integers. For
$\ell\in\Z_0^+$ the subset $\Delta_\ell=\{(\cdot,\ell)\in\Delta\}$ of $\Delta$ is called its
{\it$\ell$th level}. By some slight abuse of notation, we let \(  \Delta_{0}  \) denote both the  subset of the Riemannian manifold \(  M  \) on which the induced map \(  f^{R}  \) is defined and the 0'th level of tower \(  \Delta  \).
The collection  $\Delta_{\ell,i}:=\{(z,\ell)\in\Delta_\ell|\,\,(z, 0)\in\Delta_{0,i}\}$
forms a partition of $\Delta$ that we denote by $\eta.$
The set $\Delta_{R_i-1,i}$ is called the {\it top
level} above $\Delta_{0, i}.$ We can then define a map \(  F:
\Delta \to \Delta \) letting
\begin{equation}\label{towerm}
F(z,\ell)=\begin{cases}
(z,\ell+1) &\text{if} \,\, \ell+1<R(z),\\
(f^{R(z)}(z), 0)&  \text{if}\,\, \ell+1=R(z).
\end{cases}
\end{equation}
There exists a natural projection  $\pi:\Delta \to M$
defined by   $\pi(x, \ell)=f^\ell(x_0)$ for $x\in\Delta$ with $F^{-\ell}(x)=x_0\in\Delta_0.$
Notice that  \(  \pi   \) is a semi-conjugacy $f\circ\pi=\pi\circ F.$

For future reference we note that we can extend the return time and the separation time to all of \(  \Delta  \). Indeed, for any $x\in\Delta$ we can define a \emph{first hitting time} by
\begin{equation}\label{hitting}
\hat{R}(x):=\min \{n\geq 0 : F^n(x)\in \Delta_0\}.
\end{equation}
Notice that if \(  x\in \Delta_{0}  \) then \(  \hat R(x) = 0  \).
We also  extend the separation time to $\Delta$ by setting $s(x, y)=0$ if $x$ and $y$ belong to different elements of $\eta$ and $s(x, y)=s(\tilde x, \tilde y)$  if $x$ and $y$ are in the same element of $\eta,$ where $\tilde x$ and $\tilde y$ are the corresponding projections to $\Delta_0.$

We define a reference measure  \(  m  \) on \(  \Delta\) as
follows. Let $\mathcal{A}$ be the Borel $\sigma$-algebra on
$\Delta_0$ and let $m_0$ denote
the restriction of Lebesgue measure to $\Delta_0$ where, as mentioned above,  we are identifying  \(  \Delta_{0}\subset M  \) with the 0'th level of the tower.  For  any $\ell\geq 0$ and $A\subset\Delta_\ell$ such that
 $F^{-\ell}(A)\in\mathcal{A}$ define
$m(A)=m_0(F^{-\ell}(A)).$ Notice that with this definition, the restriction of \(  m  \) to \(  \Delta_{0}  \) is exactly \(  m_{0}  \), whereas the restriction of \(  m  \) to the upper levels of the tower is not equal to the Riemannian volume of their projections on \(  M  \) because the tower map \(  F: \Delta \to \Delta   \) is by definition an isometry between one level and the next, except on the top level where it maps back to the base with an expansion which corresponds to the ``accumulated'' expansion of all the iterates on the manifold.
Correspondingly, for every $x\in \Delta_{R_i-1, i}\subset\Delta$, the Jacobian of \(  F  \) at \(  x  \) is  $JF(x)=\text{det} Df^R(x_0),$ where  $x_0=F^{-R_i+1}x \in\Delta_0,$  and $JF(x)= 1$ otherwise.

\section{A tower for the product}\label{tfprod}
In this section we begin the proof of our main result. To simplify the notation we will assume that we have a product of only two systems, the general case follows immediately by iterating the argument.
\subsection{Basic ideas and notation}
We begin by introducing some basic notions. Suppose  $F:(\Delta, m)\circlearrowleft$ is a GMY-tower as defined in
\eqref{tower1} and \eqref{towerm} above  and let
 $\eta$ be the partition of $\Delta$ into $\Delta_{\ell, i}$'s. Then,  for \(  n\geq 1  \),   let
\be\label{join}
\eta_n:=\bigvee_{j=0}^{n-1}F^{-j}\eta:=\{A_1\cap F^{-1}(A_2)\cap ...
\cap F^{1-n}(A_n)| \,\, A_1, ..., A_n\in \eta\}
\ee
be the refinements of the partition \(  \eta  \) defined by the map \(  F  \).
For  $x\in \Delta$ let $\eta_n(x)$ be the element containing $x.$ From \eqref{join} it is easily seen that $\eta_n(x)$ has the form
$$
\eta_n(x)=\left(\bigvee_{j=0}^{n-1}F^{-j}\eta\right)(x)=\eta(x)\cap F^{-1}\eta(F(x))\cap...\cap F^{1-n}\eta(F^{n-1}(x)).
$$

\begin{remark}
To get a better feeling for the partitions \(  \eta_{n}  \) notice that
from the definition of tower for $x\notin\Delta_0$ we have
$F^{-1}(\eta(x))=\eta(F^{-1}(x)),$ which shows that the element $\eta(x)$
gets refined only when $F^{-j}(x)\in\Delta_0$ for some $j,$
$j=0,...,n-2.$  It may be instructive to consider more in detail the cases \(  n=2  \) and \(  n=3  \) (notice that \(  \eta_{1}= \eta  \)).
For  $n=2$, from \eqref{join} we have
$F^{-1}\eta\vee\eta.$ In this case only the elements on the top
levels (recall definition just before equation \eqref{towerm}) get refined so that the new elements are mapped by
$F$ bijectively onto $\Delta_{0,i}\subset\Delta_0,$ for some $i$.
\begin{figure}[!ht]
\centering{
   \mbox{\subfigure{\includegraphics[scale=0.35]{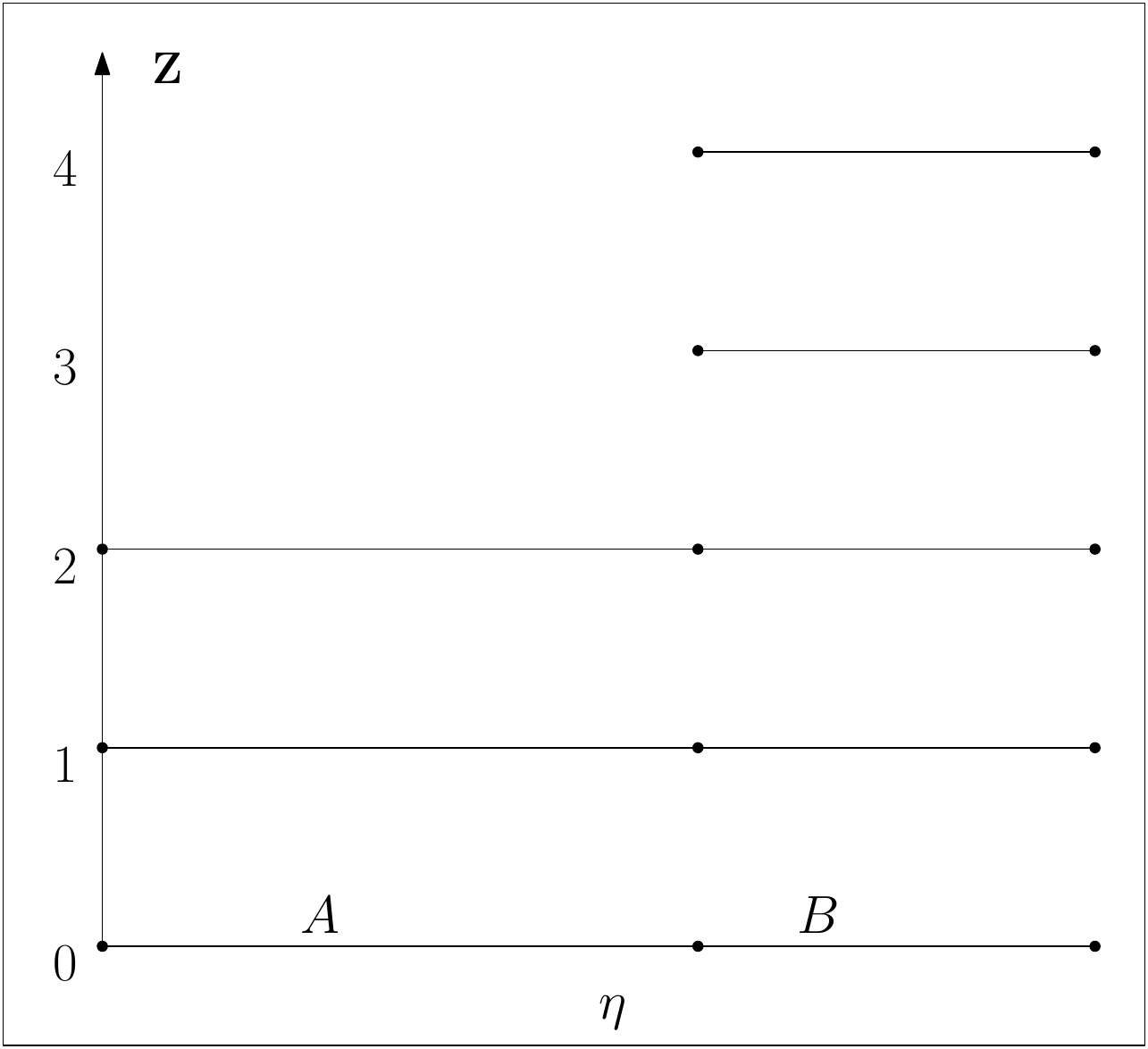}}\qquad\qquad
 \subfigure{\includegraphics[scale=0.35]{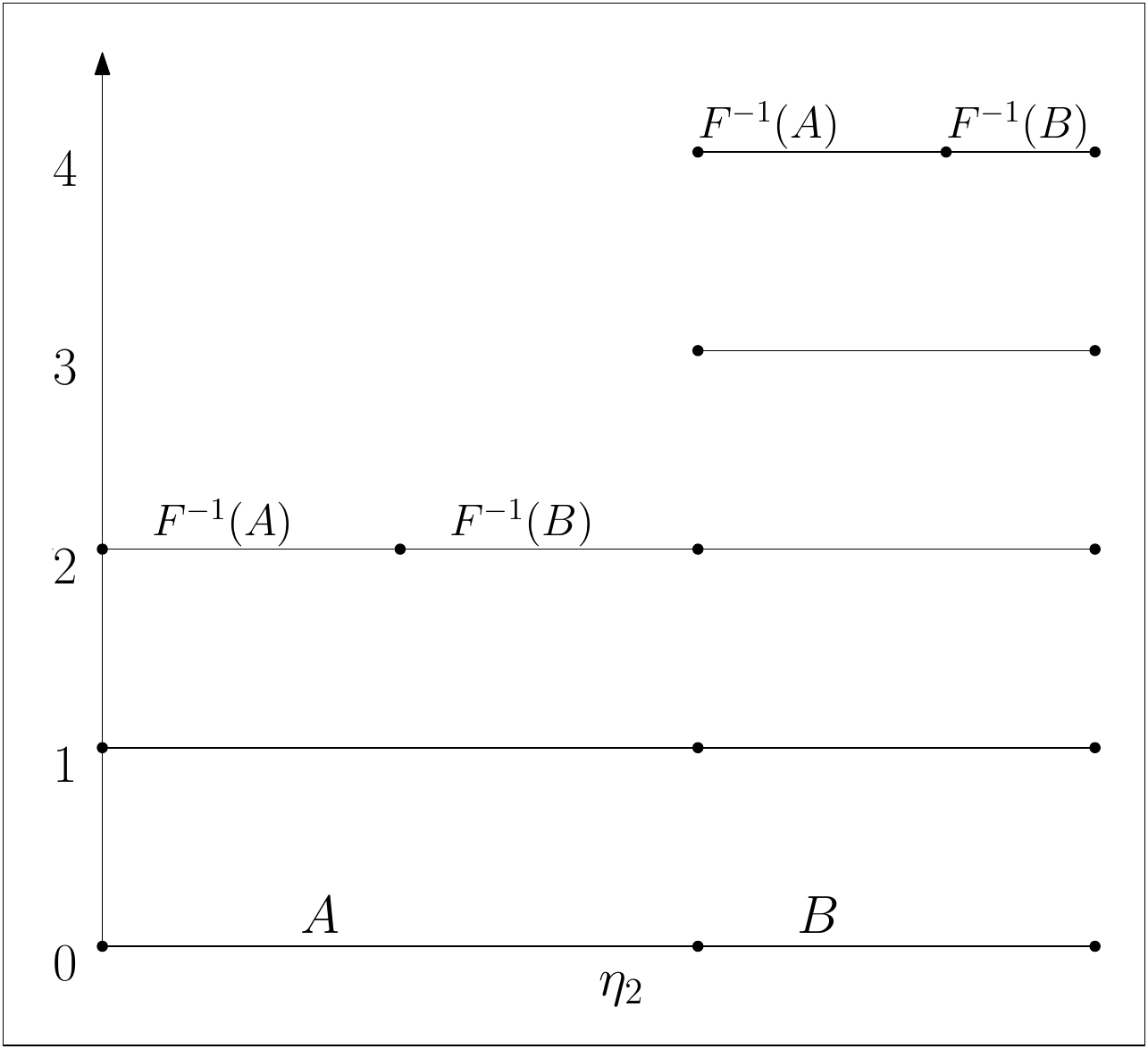}}}}
    \caption{$\eta$ and $\eta_2$}\label{pic1}
\end{figure}
All the other elements remain unchanged, see Figure \ref{pic1} (for simplicity, the
pictures are drawn when the partition of base contains only two
elements and has return times 3 and 5, in
particular, for $x\in\Delta$ with $F^2(x)\in\Delta_0$  we have
$F^2(\eta_2(x))=\Delta_0.$ This is because $\eta_2(x)=\eta(x)$ and
$F(\eta(x))$ is top level element of $\eta).$
\begin{figure}[!ht]
   \centering{
    \mbox{\subfigure{\includegraphics[scale=0.35]{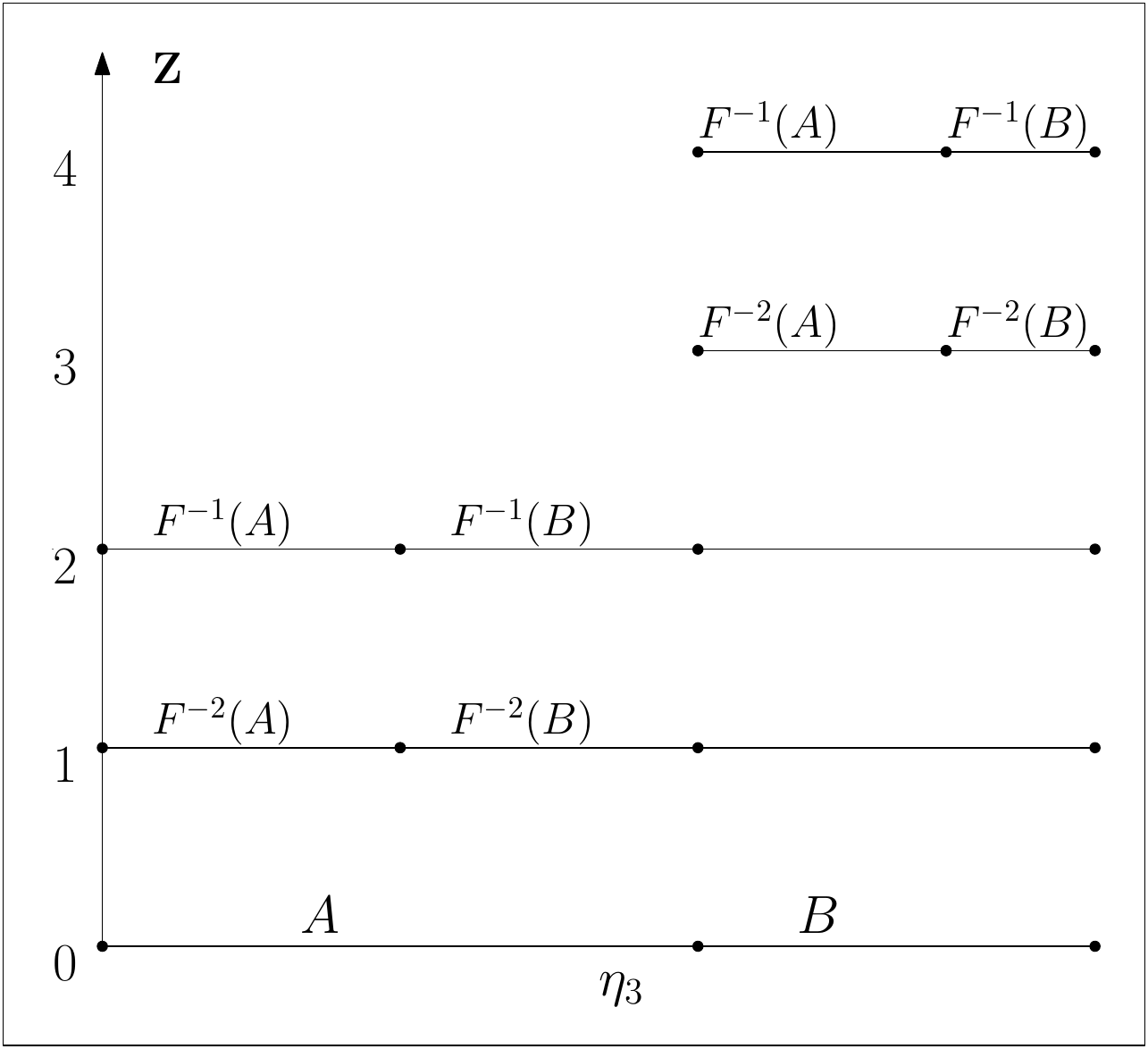}}\qquad\qquad
    \subfigure{\includegraphics[scale=0.35]{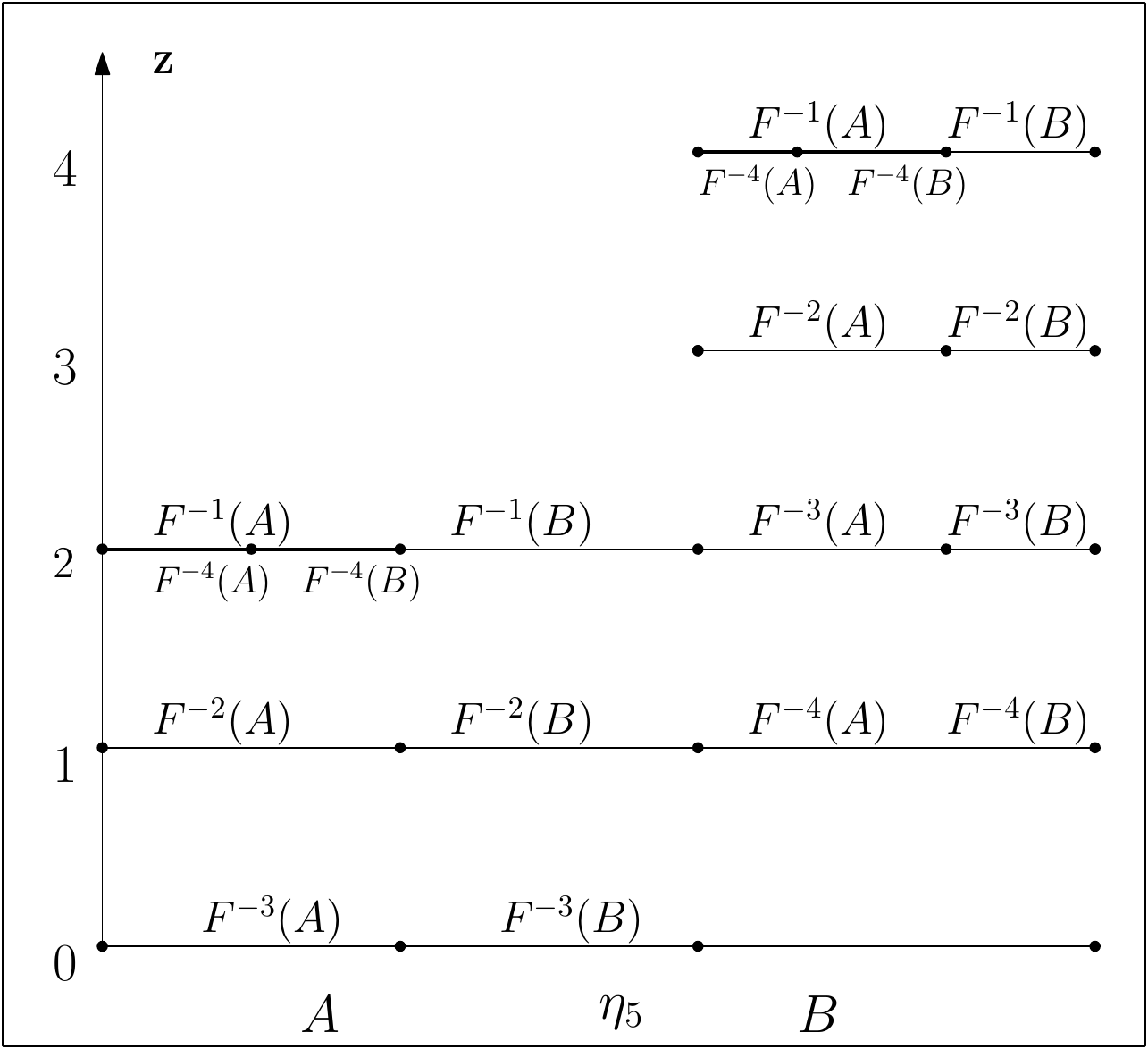}}}}
    \caption{$\eta_3$ and $\eta_5$}\label{pic2}
  \end{figure}
For $n=3$ all the elements on the top levels and on the levels ``immediately below'' the top levels
get refined so that the top level elements of the new partition
are mapped onto some $\Delta_{0,i}$ by $F;$  the elements belonging
to the levels immediately below the top levels  are mapped onto some
$\Delta_{0,i}$ by $F^2$ and other elements remain unchanged, see the left hand side of Figure
\ref{pic2}. On the right hand side of Figure \ref{pic2} is illustrated the situation for \(  n=5  \) in this simple example, where   $n=\min_i\{R_i\}+2,$ and therefore the refinement procedure ``reaches'' $\Delta_0.$  After
this time, the top levels undergo a ``second round'' of refinements. In
Figure  \ref{pic2} the bold elements of $\eta_5$ are the elements of
$\eta$ which have been refined twice.
In general, for each $n$ the refinement procedure affect
$n-1$ levels below the top levels. If $n$ is sufficiently large, some of the partition elements might get refined several times.
\end{remark}

%\begin{proof}
%The fact $\text{gcd}\{R_i\}=1$ implies $\mu$  mixing was proved in \cite{LSY} (Theorem 1).
%
%Here we prove the fact  if  $\mu$ is mixing then $\text{gcd}\{R_i\}=1.$
%Assume by contradiction $\mu$ is mixing and $\text{gcd}\{R_i\}=k>1.$ By the remark \ref{foraper}
%for $A\subset\Delta_{0,i}$ the relation $F^{-n}(A)\cap\Delta_0\neq \emptyset$ holds
%only when $n$ is divisible by $k,$ which contradicts the mixing property of $\mu.$
%\end{proof}
%The lemma which we have just proved shows that the condition
%gcd$\{R_i\}=1$ is equivalent definition of aperiodicity for towers

The following statement follows almost immediately from the observation above. Recall also the definition of the first hitting time in \eqref{hitting}.

 \begin{lemma}\label{abtjoin}
For any $x\in F^{-n}(\Delta_0)$ the map $F^n:\eta_n(x)\to \Delta_0$ is a bijection and   $F^n(\eta(x))=\Delta_0.$
 Moreover if $\hat{R}(x)=n$ then $\hat{R}|\eta_n(x)\equiv n.$
 \end{lemma}

\begin{proof}
The proof is by induction on $n.$
For $n=1$ we have no refinement, $\eta_1=\eta$, and therefore the conclusion
follows from the definition of tower.
Assume that the assertion is true for $n=k.$ From equality \eqref{join} we obtain
$\eta_{k+1}=F^{-1}(\eta_k)\vee\eta.$
 Let $x\in\Delta$ be a point, such that $F^{k+1}(x)\in\Delta_0$ then using the relation
 $\eta_{k+1}(x)=F^{-1}\eta_k(F(x))\cap\eta(x)$ we obtain
$F^{k+1}(\eta_{k+1}(x))=F^k(\eta_k(F(x)))\cap F^{k+1}(\eta(x)).$
Since $F^k(\eta_k(F(x)))=\Delta_0$  from the inductive assumption
and $\Delta_0\subset F^{k+1}(\eta(x))$
 from the definition of tower we get $F^{k+1}(\eta_{k+1}(x))=\Delta_0.$
Since the first return time is constant on the elements of $\eta,$ the
second assertion follows.
\end{proof}

\subsection{Return times to  $\bar{\Delta}_0$}\label{partition}
We are now ready to begin the construction of the tower for the product system. Since we are considering the product of just two systems, we will omit superfluous indexing and
let
$f$ and $f'$ be two maps that admit GMY-structure with the bases $\Delta_0,$ $\Delta_0'$ and return time functions $R,$ $R'$ respectively.
Then we have associated GMY-towers
 \[
 F:(\Delta, m)\circlearrowleft
 \quad \text{ and } \quad
 F':(\Delta', m')\circlearrowleft
 \]
with  bases $\Delta_0$ and  $\Delta_0'$ and return time functions
\(   R(x)  \) and  $ R'(x')$ respectively.
Let
\[  \bar{\Delta}=\Delta\times\Delta'
\quad \text{ and } \quad
\bar{\Delta}_0=\Delta_0\times\Delta_0'
\]
denote the product of the two towers, and
 the product of their bases respectively. Letting \(  m_{0}, m_{0}'  \) denote the restrictions of \(  m, m'  \) to \(  \Delta_{0}, \Delta_{0}'  \) respectively, we let  \(  \bar m = m \times m'  \) and \(  \bar m_{0} = m_{0} \times m'_{0}  \)  denote the product measures on the corresponding products.
 The direct product map    $\bar{F}=F\times F'$ is defined on \(  \bar \Delta  \) and we will construct a tower for \(  \bar F  \) with base \(  \bar{\Delta}_0  \).
We start by defining the return time function \(  T  \) on  $\bar\Delta_0$.  From Theorem 1 in
\cite{LSY}  there exist mixing invariant  probability measures $\mu,$ $\mu'$ for $F$ and $F'$, equivalent to \(  m, m'  \) respectively, and with densities which are uniformly bounded above and below.
 Therefore,
there exist constants \(  c>0  \) and  $n_0>0$  such that
\begin{equation}\label{mixing}
m(F^{-n}(\Delta_0)\cap\Delta_0)>c>0
\quad\text{ and } \quad
m'(F'^{-n}(\Delta_0')\cap\Delta_0')>c>0
\end{equation}
for all \( n\geq n_0\).
 We choose such $n_0$
and introduce  a sequence $\{\tau_i\}$ of positive integers as follows.
For
$\bar{x}=(x, x')\in\bar{\Delta}$  let
$$
\tau_0(\bar{x})=0 \quad \text{ and } \quad \tau_1(\bar{x}):=n_0+\hat{R}(F^{n_0}x).
$$
The other elements of the sequence \(  \{\tau_{i}\}  \) are defined inductively  by iterating  \(  F  \) or \(  F'  \)  alternately depending on whether \(  i  \) is odd or even. More formally,  for every  $j\ge 1$ let
\begin{equation}\label{tau}
\begin{aligned}
&\tau_{2j}(\bar{x}) :=\tau_{2j-1}(\bar{x})+n_0+\hat{R}'({F'}^{n_0+\tau_{2j-1}}x'),
\\
&\tau_{2j+1}(\bar{x}) :=\tau_{2j}(\bar{x})+n_0+\hat{R}(F^{n_0+\tau_{2j}}x).
\end{aligned}
\end{equation}

\begin{remark}\label{n0}
Notice that at every step we ``wait'' for \(  n_{0}  \) iterates before defining the next term of the sequence.
This implies that  for any $i$ we have $\tau_i-\tau_{i-1}\ge n_0$
and therefore, from \eqref{mixing}, we get
\[
m_0(\Delta_0\cap F^{\tau_{i-1}-\tau_i}\Delta_0)\geq c>0
\quad\text{ and } \quad
m_0(\Delta_0\cap F^{\tau_{i-1}-\tau_i}\Delta_0)\geq c>0
\]
We use this fact in the proof of the first item in Proposition \ref{key}.
\end{remark}

\subsection{Initial step of construction of the partition of \(  \bar\Delta_{0}  \)}
Now we can begin to define a partition $\bar{\eta}$ of $\bar{\Delta}_0.$
Recall that the towers  $\Delta,$  $\Delta'$ of the component systems admit by definition partitions into sets of the form
 $\Delta_{\ell,j},$ $\Delta'_{\ell,j}$. We will denote these given partitions by \(  \eta, \eta'  \) respectively and their restrictions to \(  \Delta_{0}, \Delta_{0}'  \) by \(  \eta_{0}, \eta_{0}'  \). We let
 \[
 \xi_{0}=\eta_{0}\times \eta_{0}'
 \]
 denote the corresponding partition of the product \(  \bar \Delta_{0}  \). Our goal is to define a partition \(  \bar\eta  \) of \(  \bar\Delta_{0}  \)  with the property that for each \(  \bar x\in \bar\Delta_{0}  \) the corresponding partition element \(  \bar\eta(\bar x) \in \bar\eta  \) maps bijectively to \(  \bar\Delta_{0}  \) with some return time \(  T(\bar\eta(\bar x))  \).
 Its construction requires the definition of
  an increasing sequence of partitions of \(  \bar\Delta_{0}  \) denoted by
$$\xi_{0}\prec\xi_1\prec\xi_2\prec \xi_3\prec...$$
A key property of these partitions will be that the sequences \(  \tau_{1},..., \tau_{i}  \) are constant on elements of \(  \xi_{i}  \). Moreover,   the construction of \(  \bar\eta  \) implies that all return times \(  T(\bar x)=T(\bar\eta((\bar x))  \) are of the form \(  T(\bar x)= \tau_{i}(\bar x)  \) for some \(  {i}  \) where \(  \tau_{i}  \) belongs to the sequence defined above. This property will be used below for the estimates of the tail of the return times.
In particular, notice also that there may be some elements of \(  \eta_{0}\times \eta_{0}'  \) which map bijectively to \(  \bar\Delta_{0}  \), and thus are candidates for elements of \(  \bar\eta  \), but are not guaranteed to satisfy the requirement just stated above.
The partition \(  \bar\eta  \) will be defined as the union
\begin{equation}\label{etabar}
\bar\eta = \bigcup_{i=1}^{\infty} \bar\eta_{i}
\end{equation}
of disjoint sets \(  \bar\eta_{i}  \) which consists of a collection of subsets which are defined in the first \(  i  \) steps of the construction.

The systematic construction proceeds as follows.
For each
$\bar{x}=(x,x')\in\bar{\Delta}_0,$  let
\begin{equation}\label{etatau}
\eta_{\tau_1(x)}(x):=(\bigvee_{j=0}^{\tau_1(\bar{x})-1}F^{-j}\eta)(x).
\end{equation}
It follows immediately that \(  \eta_{\tau_{1}}(x)\subseteq \eta_{0}(x)  \).  Indeed, as the following simple Lemma proves, collection $\eta_{\tau_1}:=\{\eta_{\tau_1(x)}(x)|_{}x\in\Delta_0\}$ is in fact a partition of \(  \Delta_{0}  \).
\begin{lemma}\label{constanttau}
\(  \eta_{\tau_{1}}  \) is a partition of $\Delta_0$, \(  \eta_{0}\prec \eta_{\tau_{1}}  \) and \(  \tau_{1}  \) is constant on elements of \(  \eta_{\tau_{1}}  \).
\end{lemma}

\begin{proof}
First of all, note  that since $\tau_1$ depends only on the first coordinate $\eta_{\tau_1(x)}(x)\in\eta_n,$ with $n=\tau_1(x).$ Then Lemma \ref{abtjoin} implies that  $F^{\tau_1(x)}(\eta_{\tau_1(x)})=\Delta_0$ bijectively. In particular, all the points in $\eta_{\tau_1(x)}(x)$ have the same combinatorics up to time $\tau_1(x),$ and hence $\tau_1$ is constant
on $\eta_{\tau_1(x)}(x).$

Now, to prove the Lemma, since   $\eta_{\tau_1}$ clearly covers $\Delta_{0}$, we just need to show every pair of sets in $\eta_{\tau_1}$ are either disjoint or  coincide. Let $\eta_{\tau_1(x)}(x)$ and  $\eta_{\tau_1(y)}(y)$ be two arbitrary elements of $\eta_{\tau_1}.$ If $\tau_1(x)=\tau_1(y)=n$ then $\eta_{\tau_1(x)}(x)$ and  $\eta_{\tau_1(y)}(y)$ are the elements of $\eta_n,$ hence they are disjoint or coincide. If $\tau_1(x)\neq\tau_1(y)$ then $\eta_{\tau_1(x)}(x)\cap\eta_{\tau_1(y)}(y)=\emptyset$ because $\tau_1$ is constant on  elements of $\eta_{\tau_1}.$
\end{proof}
We can now define the partition $\xi_1$ of $\bar\Delta_0$ by letting, for every
\( \bar x=(x, x')\in\bar\Delta_0  \),
$$
\xi_1(\bar{x}):=\eta_{\tau_1(x)}(x)\times \eta_{0}'(x').
$$
Notice  that each element \(  \Gamma\in \xi_{1}  \)  has an associated value of \(  \tau_{1}  \) such that \(  F^{\tau_{1}}(x)\in \Delta_{0}  \) for every
\(  x \in \pi\Gamma  \). On the other hand, we do not a priori have any information about the location of  \(  F'^{\tau_{1}}(x')  \) for \(  x' \in \pi' \Gamma  \) (since \(  \tau_{1}  \) is defined in terms of properties of  \(  F  \)). To study the distribution of such images, for a given \(  \Gamma\in \xi_{1}  \), we consider sets of the form
\begin{equation}\label{eta1}
\eta'_{\tau_{1}}(x')=(\bigvee_{j=0}^{\tau_1-1}F'^{-j}\eta')(x').
\end{equation}
The collection of such sets, for all \(  x'\in \pi'\Gamma  \),  form a refinement of \(  \pi'\Gamma  \).   For those points \(  x'\in \pi'\Gamma  \) such that \(  F^{\tau_{1}}(x')\in \Delta_{0}'  \) we then have that
\[
F'^{\tau_{1}}: \eta'_{\tau_{1}}(x') \to \Delta_{0}'
\]
bijectively.  In this case we consider the set
\[
\bar\eta_{1}(x, x') = \eta_{\tau_{1}}(x)\times \eta'_{\tau_{1}}(x').
\]
which maps bijectively to \(  \bar\Delta_{0}  \) by \(  \bar F^{\tau_{1}}  \),
and let \(  \bar\eta_{1}  \) denote the collection of all sets of this form constructed at this step. This is the first collection of sets which will be included in the union \eqref{etabar} defined above.

\subsection{General step} We now describe the general inductive step in the construction of the sequence of partitions \(  \xi_{i}  \) and the sets \(  \bar\eta_{i}  \).
The main inductive assumption is that partitions
$\xi_i$ of \(  \bar\Delta_{0}  \) have been  constructed for all $i<k$   in such a way that  on each element of $\xi_i$ the functions $\tau_1,..., \tau_{i}$ are constant and such that each component of  \(  \bar\eta_{i-1}  \) is contained inside an element of \(  \xi_{i-1}  \) and contains one or more elements of \(  \xi_{i}  \) (in particular elements of the partition \(  \xi_{i}  \) either have empty intersection with \(  \bar\eta_{i-1}  \) or are fully contained in some component of \(  \bar\eta_{i-1}  \)).

The notation is slightly different depending on whether \(  k  \) is even or odd, according to the different definitions of \(  \tau_{k}  \) in these two cases, recall \eqref{tau}.  For definiteness we assume that \(  k  \) is odd, the construction for \(  k  \) even is the same apart from the change in the role of the first and second components.  We fix some \(  \Gamma\in \xi_{k-1}  \) and define the partition \(  \xi_{k}|\Gamma  \) as follows.
For $\bar x\in\Gamma$,  let
\[
\eta_{\tau_{k}(x)}(\bar x)=(\bigvee_{j=0}^{\tau_{k}(\bar{x})-1}F^{-j}\eta)(x).
\]
A direct generalization of Lemma \ref{constanttau} gives  that the collection of sets \(  \eta_{\tau_{k}}:=\{\eta_{\tau_{k}(x)}(x)|x\in\pi\Gamma \} \) is a partition of \(  \pi\Gamma  \) on whose elements \(  \tau_{k}  \) is constant. For every $\bar x\in\Gamma$ we let
$$
\xi_{k}(\bar{x}):=\eta_{\tau_{k}(x)}(\bar x)
\times\pi'\Gamma.
$$
This completes the definition of the partition \(  \xi_{k}  \) and allows us define \(  \bar\eta_{k}  \).
As mentioned in the inductive assumptions above, each component of \(  \bar\eta_{k}  \) will be contained in an element of \(  \xi_{k}  \). Thus, generalizing the construction of such elements in the first step given above, we fix one element \(  \Gamma\in \xi_{k}  \) and proceed as follows.
By construction we have $F^{\tau_{k}}(\pi\Gamma)=\Delta_0$ and ${F'}^{\tau_{k}}(\pi'\Gamma)$ is spread around $\Delta'.$  Therefore, for every $x'\in \pi'\Gamma$  we consider sets of form
\begin{equation}\label{etaprime}
\eta'_{\tau_{k}}(x')=
(\bigvee_{j=0}^{\tau_{k}-1}F'^{-j}\eta')(x').
\end{equation}
The collection of such sets, for all \(  x'\in \pi'\Gamma  \), form a refinement of \(  \pi'\Gamma  \).
For those points \(  x'\in \pi'\Gamma  \) such that \(  F^{\tau_{k}}(x')\in \Delta_{0}'  \) we then have that
\[
F'^{\tau_{k}}: \eta'_{\tau_{k}}(x') \to \Delta_{0}'
\]
bijectively.  In this case we consider the set
\[
\bar\eta_{k}(x, x') = \eta_{\tau_{k}}(x)\times \eta'_{\tau_{k}}(x')
\]
which maps bijectively to \(  \bar\Delta_{0}  \) by \(  \bar F^{\tau_{k}}  \),
and let \(  \bar\eta_{k}  \) denote the collection of all sets of this form constructed at this step. Moreover, for each such set we let \(  T=\tau_{k}  \). Finally, notice that the construction of  \(  \bar\eta_{k}  \) through the formula \eqref{etaprime} implies that the elements  of the partition \(  \xi_{k+1}  \), to be constructed in the next step, are either disjoint from \(  \bar\eta_{k}  \) or contained components of \(  \bar\eta_{k}  \). Indeed, the construction of \(  \xi_{k+1}  \) involves a formula analogous to \eqref{etaprime} with \(  \tau_{k}  \) replaced by \(  \tau_{k+1}  \), clearly yielding a finer partition of \(  \pi'\Gamma  \).

This completes the general step of the construction and in particular allows us to define the set \(  \bar\eta  \) as in \eqref{etabar}.
We remark that by construction all components of \(  \bar\eta  \) are pairwise disjoint but we have not yet proved that \(  \bar\eta  \) is a partition of \(  \bar\Delta_{0}  \). This will be an immediate consequence of Proposition \ref{key}  in the next section. For formal consistency of notation we let
\(
T=\infty
\)
 on all points in the complements of the elements of \(  \bar\eta  \).
Notice that it follows from the construction that for any \(  \Gamma\in \bar\eta  \):
\begin{enumerate}
\item
$T|\Gamma=\tau_i|\Gamma$ for some $i.$
\item  $\bar{F}^T(\Gamma)=\bar{\Delta}_0$.
\end{enumerate}
In particular all elements of \(  \bar\eta   \) satisfy properties (G1) and (G2) in the definition of Young Tower.
In Section \ref{asymp} we obtain some preliminary estimates concerning the general asymptotics of the return times of the elements of \(  \bar\eta  \) defined above, and in Section  \ref{pfnew} we consider the specific cases of polynomial, stretched exponential and exponential decay rates. In Section \ref{final} we use all these estimates to prove that \(  \bar\eta  \) is indeed a partition of \(  \bar\Delta_{0}  \) and that the return map \(  \bar F^{T}  \) satisfies the required properties (G3)-(G5).

\section{Return time asymptotics}\label{asymp}
 In this section we begin the study of the
asymptotics  of the return time \(  T  \). We will use the following
general notation for conditional measures: if \(  \mu  \) is a measure we write
$\mu(B|A):={\mu(A\cap B)}/{\mu(B)}.$ Also, for every \(  n\geq 1  \), we write
\begin{equation}\label{mbar1}
\bm_n:=\sum_{j\ge n}\mathcal{M}_j
\end{equation}
where \(  \mathcal M_{j}  \) is the bound on the tails of the component systems, as in \eqref{M}.
The main result of this section is the following

\begin{proposition}\label{key}
There exist constants $\varepsilon_0,$ $K_0>0$ such that for any $i\geq 2$
\begin{enumerate}
\item  $ \bar{m}_0\{T=\tau_i|T>\tau_{i-1}\}\ge\varepsilon_0$
\item  $\bar{m}_0\{\tau_{i+1}-\tau_i\geq n|\Gamma\}\leq K_0 \bm_{n-n_{0}},$
for any  $n>n_0$ and
    $\Gamma\in\xi_i.$
\end{enumerate}
\end{proposition}
Recall that we have set \(  T=\infty  \) on the complement of points belonging to some element of \(  \bar\eta  \), and notice that \(  \{T>\tau_{i-1}\} \) is the set of points in \(  \bar\Delta_{0}  \) which do not belong to any elements of \(  \eta_{j}  \) for any \(  j=1,..., i-1  \) (by some slight abuse of notation we could write this as \(   \{T>\tau_{i-1}\}  =\bar\Delta_{0}\setminus \bigcup_{j=1}^{i-1}\bar\eta_{j}   \)). As an immediate consequence of item (1) we get the statement that \(  T  \) is finite for almost every point in \(  \bar\Delta_{0}  \) and therefore the collection of sets \(  \bar\eta  \) as in \eqref{etabar} is indeed a partition of \(  \bar\Delta_{0}  \) mod 0.

The proof of Proposition \ref{key} relies on some standard combinatorial estimates which, for completeness and to avoid interrupting the flow of the calculations,  we include in Appendix \ref{appendixA}.

\begin{proof}[Proof of (1)]
By construction,  $\{T>\tau_{i-1}\}$ is a union of
elements of the partition $\xi_i.$ Thus
  $$\bar{m}_0\{T=\tau_i|T>\tau_{i-1}\}=\frac{1}{\bar{m}_0\{T>\tau_{i-1}\}}
 \sum_{\Gamma\in\xi_i, T|_{\Gamma}>\tau_{i-1}}\bar{m}_0\{\{T=\tau_i\}\cap\Gamma\}.$$
Thus it is sufficient to prove
 $\bar{m}_0\{T=\tau_i|\Gamma\}\geq \varepsilon_0$
 for any $\Gamma\in\xi_i$ on which $T|_{\Gamma}>\tau_{i-1}.$
Assume for a moment $i$ is even and let $\Omega =\pi(\Gamma),$ $\Omega'=\pi'(\Gamma).$ Then by construction
$\pi'(\{T=\tau_i\}\cap\Gamma)=\Omega'$ and
\[
 \bar{m}_0\{T=\tau_i|\Gamma\}=\frac{\bar{m}_0(\{T=\tau_i\}\cap \Gamma)}{\bar{m}_0(\Gamma)}=
 \frac{m_0(\Omega\cap F^{-\tau_i}\Delta_0)m_0'(\Omega')}{m_0(\Omega)m_0'(\Omega')}=
 \frac{m_0(\Omega\cap F^{-\tau_i}\Delta_0)}{m_0(\Omega)}.
\]
Now recall that  $F^{\tau_{i-1}}(\Omega)=\Delta_0$, which implies
$\Omega\cap F^{-\tau_{i-1}}\Delta_0=\Omega$ and therefore
\[
 \frac{m_0(\Omega\cap F^{-\tau_i}\Delta_0)}{m_0(\Omega)} =\frac{m_0(\Omega\cap F^{-\tau_{i-1}}
 (\Delta_0\cap F^{\tau_{i-1}-\tau_i}\Delta_0))}{m_0(\Omega)}  =
 F^{\tau_{i-1}}_\ast(m_0|\Omega)(\Delta_0\cap F^{\tau_{i-1}-\tau_i}\Delta_0).
\]
Notice that  $m_0(\Delta_0\cap F^{\tau_{i-1}-\tau_i}\Delta_0)\geq c>0$ since
 $\tau_i-\tau_{i-1}\geq n_0.$
Letting $\nu=F_\ast^{\tau_{i-1}}(m_0|\Omega),$ applying Corollary \ref{dist} with $n=\tau_{i-1}$ for $x,y\in \Delta_0$ we get
\begin{equation*}\begin{split}
\bar{m}_0\{T=\tau_i|\Gamma\}= \nu(\Delta_0\cap F^{\tau_{i-1}-\tau_i}\Delta_0)\geq
\frac{m_0(\Delta_0\cap F^{\tau_{i-1}-\tau_i}\Delta_0)}
{(1+D)m_0(\Delta_0)}\ge \frac{c}{(1+D)m_0(\Delta_0)}.
\end{split}
\end{equation*}
For odd $i$'s we can just change $F$ to $F'$ and  do all the calculations, that gives the estimate
\begin{equation*}
\bar{m}_0\{T=\tau_i|\Gamma\}>\frac{c}{(1+D')m_0'(\Delta_0')}.
\end{equation*}
Taking $\varepsilon_0=c
\min\{\frac{1}{(1+D)m_0(\Delta_0)}, \frac{1}{(1+D')m_0'(\Delta_0)'}\}$
 we get the assertion.
   \end{proof}

\begin{proof}[Proof of (2)]
Assume for a moment $i$ is even and let, as above, $\Omega=\pi(\Gamma),$
$\Omega'=\pi'(\Gamma).$  Since $\tau_i$ is
constant on the elements of $\xi_i$ we have
$$\pi\left(\{\bar{x}=(x,x')|\hat{R}\circ F^{\tau_i+n_0}(x)>n\}\cap\Gamma\right)=
\{x|\hat{R}\circ F^{\tau_i+n_0}(x)>n\}\cap\Omega.$$
For convenience we begin by estimating \(  \bar{m}_0\{\tau_{i+1}-\tau_i-n_0>n|\Gamma\}   \).
From the definition of $\tau_i$, letting $\nu=F^{\tau_{i-1}}_\ast(m_0|_{}{\Omega})$,  we have
 \begin{align*}
  \bar{m}_0\{\tau_{i+1}-\tau_i-n_0>n|\Gamma\}
  &=\bar{m}_0\{\hat{R}\circ F^{\tau_i+n_0}>n|\Gamma\}
=\frac{\bar{m}(\{\hat{R}\circ F^{\tau_i+n_0}>n\}\cap \Gamma)}{\bar{m}(\Gamma)}
 \\ &=\frac{m_0'(\Omega')m_0(\pi\{\hat{R}\circ F^{\tau_i+n_0}>n\}\cap\Omega)}
  {m_0(\Omega)m_0'(\Omega')}
  \\
  &=
  \frac{m_0(\{\hat{R}\circ F^{\tau_i+n_0}>n\}\cap\Omega)}{m_0(\Omega)}
=m_0\{\hat{R}\circ F^{\tau_i+n_0}>n|\Omega\}
  \\& = F^{\tau_i+n_0}_\ast m_0\{\hat{R}>n|\Omega\}
  =F^{\tau_i-\tau_{i-1}+n_0}_\ast \nu\{\hat{R}>n\}.
 \end{align*}
 To bound the final term in terms of
 \(  m_0\{\hat{R}>n\}  \) it is sufficient to show that the density of \(  F^{\tau_i-\tau_{i-1}+n_0}_\ast \nu \) with respect to \(  m_{0}  \) is uniformly bounded in \(  i \).  We write first
\begin{equation*}
\frac{dF^k_\ast\nu}{dm}(x)=
\frac{dF^k_\ast\nu}{dF^k_\ast m_{0}}
\frac{dF^k_\ast m_{0}}{dm}(x)
=
\sum_{x_0\in F^{-k}(x)}\frac{\frac{d\nu}{dm_0}(x_0)}{JF^n(x_0)}\leq
\left\|\frac{d\nu}{dm_0}\right\|_\infty \left\|\frac{dF^k_\ast m_0}{dm_0}\right\|_\infty.
\end{equation*}
The second factor is bounded by $M_0$ from Lemma \ref{m0}. Let
us estimate the first one. Note that,
$\frac{d\nu}{dm_0}(x)=\frac{1}{JF^{\tau_{i-1}}x_0},$ where
$x_0=(F^{\tau_{i-1}}|\Omega)^{-1}(x).$  Corollary  \ref{dist} implies that
${\left\| {d\nu}/{dm}\right\|}_\infty
\leq  (1+D)/{m_0(\Delta_0)}$
and so we get
$$
\bar{m}_0\{\tau_{i+1}-\tau_i>n_0+n|\Gamma\}=  \bar{m}_0\{\tau_{i+1}-\tau_i-n_0>n|\Gamma\}
\le M_0\frac{1+D}{m_0(\Delta_0)}m_0\{\hat{R}>n\}.
$$
For \(  n>n_{0}  \), and using the definition of \(  \hat R  \), we can write this is
$$
\bar{m}_0\{\tau_{i+1}-\tau_i>n|\Gamma\}
\le M_0\frac{1+D}{m_0(\Delta_0)}m_0\{\hat{R}>n-n_{0}\} =
M_0\frac{1+D}{m_0(\Delta_0)} \sum_{i\ge n-n_{0}}m_0\{R>i\}.
$$
For $i$ odd the calculation is exactly the same and we get
$$
\bar{m}_0\{\tau_{i+1}-\tau_i>n|\Gamma\}
\le M_0'\frac{1+D'}{m_0'(\Delta_0')}m_0'\sum_{i\ge n-n_{0}}m_0'\{R'>i\}.
$$
Letting
$
K_0=\max\left\{\frac{M_0(D+1)}{m_0(\Delta_0)},
\frac{M_0'(D'+1)}{m_0'(\Delta_0')}\right\}
$
and using \eqref{M} we get the assertion.
 \end{proof}

%5th section % 555555555555555555555555555555555555555555555555555555555555555

\section{Rates of decay}\label{pfnew}
We now fix some arbitrary \(  n\geq 1  \) and estimate $\bar{m}_0\{T > n\}$. Letting $\tau_0=0$ we write
 \be\label{tn} \bar{m}_0\{T>n\}=\sum_{i\geq 1}\bar m_0\{T>n; \tau_{i-1}\le n
<\tau_i\}.
\ee
We will estimate the right hand side of  \eqref{tn}  using different arguments depending on whether the decay of \(  \mathcal M_{n}  \) is exponential, stretched exponential or polynomial.

\subsection{Polynomial case} We suppose that  $\mathcal M_n=\mathcal O(n^{-\alpha})$ for some $\alpha>2$ and prove that   \be\label{eq:Tn}
\bar{m}_0\{T>n\}=\mo(n^{1-\alpha})
\ee
Let $K= 2\max\{\bar{m}_0(\bar{\Delta}_0), K_0, m_0(\Delta_0), m'(\Delta_0')\}$ and \(  \varepsilon_{0}>0  \) given by Proposition \ref{key}. We start with the following somewhat unwieldy estimate.

\begin{proposition}\label{estm} For  any  $n\in\mathbb{N}$
\begin{equation}\label{eq:estm}
\bar{m}_0\{T>n\}\leq K\sum_{i\leq\frac{1}{2}\left[\frac{n}{n_0}\right]}i(1-\varepsilon_0)^{i-3}
\bm_{\left[\frac{n}{i}\right]-n_0}
+\bar{m}_0(\bar{\Delta}_0)(1-\varepsilon_0)^{\frac{1}{2}\left[\frac{n}{n_0}\right]-1}.
\end{equation}
\end{proposition}

Before proving Proposition \ref{estm} we show how it implies \eqref{eq:Tn}.
\begin{proof}[Proof of Theorem in the polynomial case assuming Proposition \ref{estm}]
By the definition of $\bm_{\left[\frac{n}{i}\right]-n_0}$ in \eqref{mbar1} and the assumption on the asymptotics of \(  \mathcal M_{n}  \), there exists a constant \(  C>0   \) such that
for every $i\leq \frac{1}{2}\left[\frac{n}{n_0}\right]$ we have
  \begin{equation*}\begin{split}
  \bm_{\left[\frac{n}{i}\right]-n_0}=
  \sum_{j\ge {\left[\frac{n}{i}\right]-n_0}}\mathcal{M}_j\le C\sum_{j\ge {\left[\frac{n}{i}\right]-n_0}}j^\alpha\le
  C\int_{\left[\frac{n}{i}\right]-n_0}^\infty x^{-\alpha}dx\\
  \le C\frac{i^{\alpha-1}}{n^{\alpha-1}}\left(\frac{n}{n-i(n_0+1)}\right)^{\alpha-1}\le
 C\frac{i^{\alpha-1}}{n^{\alpha-1}}.
 \end{split}
 \end{equation*}
Substituting this into the statement of Proposition  \ref{estm} we get
\be\label{eq:m0Tn}
\bar{m}_0\{T>n\}\leq  \frac{KC}{n^{\alpha-1}}\sum_{i\leq\frac{1}{2}\left[\frac{n}{n_0}\right]}
i^\alpha(1-\varepsilon_0)^{i-3}+\bar{m}_0(\bar{\Delta}_0)(1-\varepsilon_0)^{
\frac{1}{2}\left[\frac{n}{n_0}\right]-1}.
\ee
 Since the series $\sum_{i=1}^{\infty}(1-\varepsilon_0)^{i-3}i^{\alpha}$ is convergent and the
 second term in \eqref{eq:m0Tn} is  exponentially small in $n,$ we get  \eqref{eq:Tn} and thus the statement of the Theorem in the polynomial case.
\end{proof}

The proof of Proposition \ref{estm} will be broken into several lemmas.
Note that
$$
\sum_{i>\frac{1}{2}\left[\frac{n}{n_0}\right]} \bar{m}_0\{T>n; \,\,
\tau_{i-1}\le n <\tau_i\}\le \bar{m}_0\{T>n;\,\,
\tau_{\frac{1}{2}\left[\frac{n}{n_0}\right]}\leq n\},
$$
which together with \eqref{tn} implies
\begin{equation}\label{Tn}
\bar{m}_0\{T>n\} \le \sum_{i\leq\frac{1}{2}\left[\frac{n}{n_0}\right]} \bar{m}_0
\{T>n;\,\, \tau_{i-1}\leq n<\tau_i\}+\bar{m}_0\{T>n;\,\,
\tau_{\frac{1}{2}\left[\frac{n}{n_0}\right]}\leq n\}.
\end{equation}

First we estimate  the second summand of \eqref{Tn}.
\begin{lemma}\label{II}
 For every $n>2n_0$  and for $\varepsilon_0>0$ as in Proposition \ref{key}
$$\bar{m}_0\{T>n;\,\, \tau_{\frac{1}{2}\left[\frac{n}{n_0}\right]}\leq n\}\leq
\bar{m}_0(\bar{\Delta}_0)(1-\varepsilon_0)^{\frac{1}{2}\left[\frac{n}{n_0}\right]-1}.$$
\end{lemma}

\begin{proof}
Since $T>n>\tau_{i-1}$ we have
\begin{align*}
\bar{m}_0\{T>n;\,\, \tau_{\frac{1}{2}\left[\frac{n}{n_0}\right]}\leq n\} \leq
\bar{m}_0\{T>\tau_{\frac{1}{2}\left[\frac{n}{n_0}\right]}\}=
\\
\bar{m}_0\{T>\tau_1\}\bar{m}_0\{T>\tau_2|\,\, T>\tau_1\}...
\bar{m}_0\{T>\tau_{\frac{1}{2}\left[\frac{n}{n_0}\right]}|\,\, T>\tau_{\frac{1}{2}\left[\frac{n}{n_0}\right]-1}\}.
\end{align*}

Notice that  $\bar{m}_0\{T>\tau_1\}\le \bar{m}_0(\bar{\Delta}_0).$
The first item of Proposition \ref{key} implies that  each of the other terms is less than $1-\varepsilon_0.$ Substituting these into the above equation  finishes the proof.
\end{proof}

Now, we begin estimating  the first summand of \eqref{Tn}.
Start with the cases $i=1, 2.$
\begin{lemma}\label{c12}
For $i=1,2$ and every $n>n_0$ we have
$$\bar{m}_0\{T>\tau_{i-1}; \tau_{i-1}\le n< \tau_i\}\le K
 \bar{\mathcal M}_{\left[\frac{n}{2}\right]-n_0}.$$
\end{lemma}
\begin{proof}
 For $i=1$ we take advantage of the fact that \(  \tau_{1}  \) depends only on the first coordinate. Then we have
\begin{equation*}\begin{aligned}\bar{m}_0\{T>0;
\tau_1>n\} &=m_0\{\tau_1>n\}m_0'(\Delta_0')
\leq  m'_0(\Delta'_0)m\{\hat{R}>n-n_0\}\\
&\leq   m'_0(\Delta'_0)
\sum_{j\ge n-n_0}m_0\{R>j\}
\leq m'_0(\Delta'_0)\bar{\mathcal M}_{n-n_{0}}
.\end{aligned}
\end{equation*}
which proves the statement in this case by the definition of \(  K  \) and using the fact that \(  \bar{\mathcal M}_{k}  \) is monotone decreasing in \(  k \).
 For $i=2$ we have
\begin{align*}
\bar{m}_0\{T>\tau_1;\tau_1\leq n <\tau_2\}
& \leq \bar m_0\{\tau_2>n\}\le \bar m_0\{\tau_2-\tau_1\ge \frac n 2\}+\bar m_0\{\tau_1\ge \frac n 2\}.
\end{align*}
From the second item of Proposition \ref{key} we have
$$
\bar m_0\{\tau_2-\tau_1\ge \frac n 2\}\le K_0 \bar{\mathcal M}_{\left[\frac{n}{2}\right]-n_0}.
$$
The second item is estimated as in the case $i=1$ and so we get
\[
\bar{m}_0\{T>\tau_1;\tau_1\leq n <\tau_2\}\leq K_0 \bar{\mathcal M}_{\left[\frac{n}{2}\right]-n_0} +
m'_0(\Delta'_0)\bar{\mathcal M}_{\left[\frac n2\right]-n_{0}}
\leq K \bar{\mathcal M}_{\left[\frac n2\right]-n_{0}}.
\]
which completes the proof in this case also.
\end{proof}

We now consider the general case.

\begin{lemma}\label{I}
For each $i\geq 3$
$$\bar{m}_0\{T>n;\,\, \tau_{i-1}\leq n< \tau_i\}\leq \sum_{j=1}^i  K(1-\varepsilon_0)^{i-3}
\bm_{\left[\frac{n}{i}\right]-n_0}.
$$
\end{lemma}

\begin{proof}
Since  $T>n\geq\tau_{i-1}$ we have
\begin{equation*}
\bar{m}_0\{T>n;\,\, \tau_{i-1}\leq n < \tau_i\} \leq
\bar{m}_0\{T>\tau_{i-1}; \,\, n<\tau_i\}.
\end{equation*}

Moreover, from
\(
\tau_i=\tau_i-\tau_{i-1}+\tau_{i-1}-\tau_{i-2}+...+\tau_1-\tau_0>n
\)
we obtain that  there is at least one $j\in[1, i]$ such that
 $\tau_j-\tau_{j-1}>\frac{n}{i}$ and therefore
 \[
\bar{m}_0\{T>n;\,\, \tau_{i-1}\leq n< \tau_i\}\leq
\sum_{j=1}^i
\bar{m}_0\{T>\tau_{i-1};\tau_j-\tau_{j-1}>\frac{n}{i}\}.
\]
 For each $i, j\ge 3$ we write
\begin{equation}\label{abc}
\bar{m}_0\{T>\tau_{i-1};\tau_j-\tau_{j-1}>\frac{n}{i}\}= Y_1\cdot Y_2\cdot Y_3
\end{equation}
where
$$
Y_1:=\bar{m}_0\{T>\tau_{j-1}; \tau_j-\tau_{j-1}>\frac{n}{i}|T>\tau_{j-2}\},$$
$$
Y_2:=\bar{m}_0\{T>\tau_1\}\bar{m}_0\{T>\tau_2|T>\tau_1\} ... \bar{m}_0\{T>\tau_{j-2}|T>\tau_{j-3}\},
$$
$$
Y_3:=\bar{m}_0\{T>\tau_j|T>\tau_{j-1};\tau_j-\tau_{j-1}>\frac{n}{i}\}...\,\,
\bar{m}_0\{T>\tau_{i-1}|T>\tau_{i-2}; \tau_j-\tau_{j-1}>\frac{n}{i}\}.
$$
 By the second item of Proposition \ref{key} we have
\be\label{Y2}
Y_2\leq \bar{m}_0(\bar{\Delta}_0)(1-\varepsilon_0)^{j-3}.
\ee
For the first term, note that
$$Y_1:=\bar{m}_0\{T>\tau_{j-1}; \tau_j-\tau_{j-1}>\frac{n}{i}|T>\tau_{j-2}\}
\leq \bar{m}_0\{\tau_j-\tau_{j-1}>\frac{n}{i}| T>\tau_{j-2}\}.$$
By  construction  $\{T>\tau_{j-2}\}$
can be written as a union of elements of $\xi_{j-1}$ and so, by the second item of Proposition \ref{key},
\be\label{Y1}
Y_1\le K_0 \bm_{\left[\frac{n}{i}\right]-n_0}.
\ee
For the third term,
since $\tau_j$ and $\tau_{j-1}$ are constant on the elements of
$\xi_j,$ if
$\tau_j(\bar{x})-\tau_{j-1}(\bar{x})>\frac{n}{i}$ for some point $\bar{x},$ then it holds on  $\xi_j(\bar{x}).$
By construction, for $k\geq j$ the partition $\xi_k$ is finer than
$\xi_j$ and  $\{T>\tau_{k-1}\}$ can be written as a union of
elements of $\xi_k.$ Hence  $\{T>\tau_{k-1};
\tau_j-\tau_{j-1}>\frac{n}{i}\}$ can be covered with elements of $\xi_k.$ Using the first item of Proposition
\ref{key} in each partition element gives
$$\bar{m}_0\{T>\tau_k| \,\, T>\tau_{k-1}; \tau_j-\tau_{j-1}>\frac{n}{i}\}\leq 1-\varepsilon_0.$$
This immediately gives
 \be\label{Y3} Y_3\leq
(1-\varepsilon_0)^{i-j}.
\ee
Substituting  \eqref{Y2}, \eqref{Y1}, \eqref{Y3} into \eqref{abc} we get the assertion of Lemma 5.5.
For the case $i\geq 3$ and $j<3$ proof will be the same but only
without $Y_2.$
\end{proof}

Notice that substituting the estimates in the statements of Lemmas
\ref{c12} and \ref{I} into   \eqref{Tn} gives the
statement in  Proposition \ref{estm}.

\subsection{Super polynomial cases}\label{sppc} In this subsection the we give the proof of the tail estimates for exponential and stretched exponential cases. Let
$$A(i)=\{\textbf{k}=(k_1, ..., k_{i-1})\in\N^{i-1}: \sum_j k_j\le n,\,\,
k_j\ge n_0, j=1, ..., i-1\},
$$
\begin{equation}\label{mbar}
\bm(i,n)=\max_{\textbf{k}\in A(i)}
\bm_{n-\sum_j k_j-n_0}\prod_{j=1}^{i-1}\bm_{k_j-n_0}.
\end{equation}

\begin{remark}\label{N}It is known fact (see for example \cite{TTao})
that the cardinality $\text{card}A(i)$ of  $A(i)$ is bounded above by $\binom{n+i-n_0}{i-1}.$
\end{remark}

Let $\delta$ be a sufficiently small number, which will be specified later. We first prove the following technical statement from which both exponential and stretched exponential cases will follow.

\begin{proposition}\label{stnexp}
For sufficiently large $n$ and any $\theta'\in (0,1]$ we have
$$\bar{m}_0\{T>n\}\le \sum_{i\le [\delta n^{\theta'}]}\binom{n+i-n_0}{i-1}
K_0^i\bm(i,n) +
\bar{m}(\bar{\Delta}_0)(1-\varepsilon_0)^{[\delta n^{\theta'}]-1}.$$
\end{proposition}

To prove Proposition \ref{stnexp} we first write, as in \eqref{Tn},
\be\label{smain}
\bar{m}_0\{T>n\}
 \le
\sum_{i\leq[\delta n^{\theta'}]}\bar{m}_0\{T>n; \tau_{i-1}\leq n <\tau_i\}+
\bar{m}_0\{T>n; \tau_{[\delta n^{\theta'}]}\leq n\}.
\ee
As in polynomial case, we start by estimating the second summand of  \eqref{smain}.
\begin{lemma}\label{exa}
For sufficiently large  $n$  we have
$$\bar{m}_0\{T>n; \tau_{[\delta n^{\theta'}]}\leq n\}\leq
 \bar{m}(\bar{\Delta}_0) (1-\varepsilon_0)^{[\delta n^{\theta'}]-1}.$$
\end{lemma}
\begin{proof}
The proof is identical to the proof of Lemma \ref{II}.
\end{proof}

To estimate the second summand of \eqref{smain},
first we fix $i$ and prove the following
\begin{lemma}\label{fixi}
For sufficiently large \(  n  \), for every \(  i \leq[\delta n^{\theta'}]  \) we have
$$\bar{m}_0\{T>n; \tau_{i-1}\le n < \tau_i\}\le \bar{m}(\bar{\Delta}_0)K_0^i
\text{card} A(i)\bar{\mathcal M} (i,n).$$
\end{lemma}

\begin{proof}
For any $\bar x\in\bar{\Delta}_0$ with $\tau_{i-1}(\bar x)\leq n <\tau_i(\bar x)$  there is
$(k_1, ..., k_{i-1})\in A(i)$ such that $k_j=\tau_j(\bar x)-\tau_{j-1}(\bar x)$ for $j=1,..., i-1$ and
$\tau_i(\bar x)-\tau_{i-1}(\bar x)>n-\sum_jk_j.$
Now, for every $\textbf{k}=(k_1, ..., k_{i-1})\in A(i)$ let
\be
P(\textbf{k},i)=\bigcap_{j=1}^{i-1}\{\bar x\in\bar{\Delta} :\,\,
\tau_j(\bar x)-\tau_{j-1}(\bar x)=k_j\} \ee and
\be
Q(\textbf{k},
i)=P(\textbf{k}, i)\cap\{\bar x\in\bar{\Delta}:\,\,
\tau_i(\bar x)-\tau_{i-1}(\bar x)>n-\sum_j k_j\}. \ee

Using the above observation and notations we can write
\be\label{summa} \bar{m}_0\{T>n; \tau_{i-1}\le n < \tau_i\}\le
\bar{m}_0\{\tau_{i-1}\le n < \tau_i\}= \sum_{\textbf{k}\in A(i)}
\bar{m}_0\{Q(\textbf{k}, i)\}. \ee
Notice that for each \(  i  \) and each \(  \textbf{k}\in A(i)  \) we have
\be \label{forj}
\begin{split}
 \bar{m}_0\{Q(\textbf{k}, i)\} = \bar{m}_0(\bar{\Delta}_0)\bar{m}_0\{\tau_1=k_1|\bar{\Delta}_0\}
\bar{m}_0\{\tau_2-\tau_1=k_2|P(\textbf{k}, 2)\}...
\\
...\bar{m}_0\{\tau_i-\tau_{i-1}>n-\sum k_j|P(\textbf{k}, i)\}.
\end{split}
\ee
Since for any $j$, the set $\{P(\textbf{k}, j)\}$ is a union of
elements of $\xi_j$,  from the second item of
Proposition \ref{key} we get
$$
\bar{m}_0\{\tau_j-\tau_{j-1}=k_j|P(\textbf{k}, j)\}\le K_0\bm_{k_j-n_0}
$$
for each $1\le j\le i-1$, and
$$
\bar{m}_0\{\tau_i-\tau_{i-1}>n-n_0-\sum k_j|P(\textbf{k}, i)\}\le K_0\bm_{n-n_0-\sum k_j}.
$$
Substituting this into  \eqref{forj} and using the definition of \(  \bar{\mathcal M} (i,n)  \) in \eqref{mbar}, we obtain
\begin{equation*}
\bar{m}_0\{Q(\textbf{k}, i)\}\le \bar{m}_0({\bar{\Delta}_0})K_0^i
\bm_{n-n_0-\sum k_j}\prod_{j=1}^{i-1}\bm_{k_j-n_0}\leq
 \bar{m}_0({\bar{\Delta}_0})K_0^i \bar{\mathcal M} (i,n).
\end{equation*}
Substituting this into \eqref{summa} completes  the proof of Lemma \ref{fixi}.
\end{proof}

\begin{proof}[Proof of Proposition \ref{stnexp}]
To prove the Proposition we just substitute the statements in the two Lemmas above into \eqref{smain} and use the upper bound for the cardinality of \(  A(i)  \), see Remark \eqref{N}.
\end{proof}

\begin{proof}[Proof of the Theorem  in the stretched exponential case] Here we consider the case
$\mathcal M_n=\mathcal{O}(e^{-\tau n^\theta})$ for some   $\tau, \theta>0$ and we show that
\begin{equation}\label{stretched}
\bar{m}_0\{T>n\}\leq Ce^{-\tau' n^{\theta'}}
\end{equation}
 for some $C_{1}, \tau'>0$ and $\theta>\theta'>0.$ We start with the following basic estimate
 \begin{lemma}\label{basic}
 For sufficiently large \(  n  \), every $i\leq [\delta n^{\theta'}]$  and every \(  \textbf{k}\in A(i)  \) we have
 \begin{equation*}
\bm_{n-\sum_j k_j-n_0}\prod_{j=1}^{i-1}\bm_{k_j-n_0}\le
{C_1}^ie^{-\tau(n-in_0)^\theta}
n^{i(1-\theta)}.
\end{equation*}
 \end{lemma}
 The statement and proof of Lemma \ref{basic} depend only on some relatively standard but non-trivial estimates concerning the tails of sequences which decay at a stretched exponential rate. In order to simplify the exposition we give the proof in the appendix.
From Lemma \ref{basic} and the definition of $\bm(i, n)$ we
obtain
\be\label{alast} \bm(i,n)\le
{C_1}^ie^{-\tau(n-in_0)^\theta}n^{i(1-\theta)}.
\ee
On the other hand combining Pascal's rule with Stirling's formula
we get
\be\label{stirling}
\binom{n+i-n_0}{i-1}\le\binom{n}{[\delta n^{\theta'}]}\le
C_2e^{\varepsilon n^{\theta'}\log n} <C_2e^{\varepsilon n^{\theta}}
\ee
for $\varepsilon>0$ such that $\varepsilon\to 0 $ as $\delta\to
0.$
Substituting inequalities \eqref{alast} and \eqref{stirling} into
Proposition \ref{stnexp} we get
\be\label{exprate}
\bar{m}\{T>n\}\le \sum_{i\le [\delta
n^{\theta'}]}C_2e^{\varepsilon n^{\theta}}K_0^i
{C_1}^ie^{-\tau(n-in_0)^\theta}n^{i(1-\theta)} +
\bar{m}(\bar{\Delta}_0)(1-\varepsilon_0)^{[\delta
n^{\theta'}]-1}.
\ee
The second term of \eqref{exprate} is of order $e^{n^{\theta'}\delta \log(1-\varepsilon_0)}.$ Hence, it remains to prove similar asymptotics for the first summand.  Notice first of all that the terms in the sum are monotone increasing in \(  i  \). Therefore denoting these terms by \(  a_{i}  \) we have that
\[
\sum_{i\le [\delta n^{\theta'}]} a_{i}
\leq
a_{[\delta n^{\theta'}]}[\delta n^{\theta'}]
= e^{\varepsilon n^{\theta}}(K_0C_1)^{[\delta n^{\theta'}]}
e^{-\tau(n-[\delta n^{\theta'}]n_0)^\theta}n^{[\delta
n^{\theta'}](1-\theta)} [\delta n^{\theta'}] .
\]
 Writing the right hand side  in exponential form, we have
 \[
a_{[\delta n^{\theta'}]}[\delta n^{\theta'}] =
\exp(\varepsilon n^\theta+\delta n^{\theta'}\log(K_0C_1)-\tau(n-n_0[\delta n^{\theta'}])^{\theta}+(1-\theta)[\delta n^{\theta'}]\log n+\log (\delta n^{\theta'})).
\]
 Factoring out \(  n^{\theta'}  \) and using the general inequality \(  a^{\theta}-b^{\theta}\leq (a-b)^{\theta}  \) the exponent is \(  \leq   \)
 \[
n^{\theta'}\left\{(\varepsilon-\tau)n^{\theta-\theta'}+\delta\log(K_0C_1)
+\tau(\delta n_0)^{\theta}n^{(\theta-1)\theta'}+(1-\theta)\delta\log n+n^{-\theta'}\log(\delta n^{\theta'})\right\}.
\]
 Since this last expression is clearly decreasing in \(  \delta  \) and negative for \(  \delta>0  \) sufficiently small we get the stretched exponential bound as required.
\end{proof}

\begin{proof}[Proof of the Theorem  in the exponential case] Here we consider the case
$\mathcal M_n=\mathcal{O}(e^{-\tau n})$ for some  $\tau>0$ and we show that
\begin{equation}\label{exp:final}
\bar{m}_0\{T>n\}\leq C e^{-\tau' n}
\end{equation}
 for some $\tau'>0$.

 \begin{lemma}\label{exponential}
 For all \(  i\leq [\delta n]  \) and every \(  \textbf{k} \in A(i)  \) we have
 $$
 \bm_{n-\sum_j k_j-n_0}\prod_{j=1}^{i-1}\bm_{k_j-n_0}\le C^ie^{-\tau(n-in_0)}.
 $$
 \end{lemma}

 \begin{proof}
Since \(  \mathcal M_{n}  \) is decaying exponentially fast, there exists \(  C'>0  \) such that  for any $j=1, ..., i-1$ we have
$$\bm_{k_j-n_0} = \sum_{\kappa\ge k_j-n_0}\mathcal{M}_\kappa \le
C'\sum_{\kappa\ge k_j-n_0}e^{-\tau \kappa}
=\frac{C'}{1-e^{-\tau}}e^{-\tau(k_j-n_0)} $$
and similarly
$$
\bm_{n-n_0-\sum k_j}\le \frac{C'}{1-e^{-\tau}}e^{-\tau(n-n_0-\sum k_j)}.
$$
Letting $C:=C'/(1-e^{-\tau})$  and substituting these estimates into the expression in the lemma finishes the proof.
 \end{proof}

By Lemma \ref{exponential} and the definition of $\bm(i,n)$ we obtain
\be\label{Min}
\bm(i,n)\le C^ie^{-\tau(n-in_0)}.
\ee
On the other hand, since  $i\le[\delta n]$,  there exists a uniform constant \(  C_{1}  \) such that
\be\label{forf}
\binom{n+i-n_0}{i-1}\le\binom{n}{[\delta n]}\le C_1e^{\varepsilon n}
\ee
for $\varepsilon>0$ such that $\varepsilon\to 0 $ as $\delta\to 0.$
 Choosing $\theta'=1$ in Proposition
\ref{stnexp} and substituting inequalities \eqref{Min} and \eqref{forf}  we obtain
\be\label{eq:expTn}
\bar{m}_0\{T>n\}\ge \sum_{i\le[\delta n]}C_1e^{\varepsilon n}(K_0C)^i e^{-\tau(n-in_0)}+
\bar{m}({\Delta_0})(1-\varepsilon_0)^{[\delta
n]-1}.
\ee
Choose $\delta$ small enough so that
$\tau'(\delta)=\tau(1-\delta n_0)-\delta \log (K_0C)-\varepsilon>0.$  This is possible, since as $\delta\to 0$ we have $\tau'(\delta)\to \tau.$
Notice, that each term of the first summand in \eqref{eq:expTn} can be written in an exponential form with exponent
\begin{align*}
\varepsilon n+i\log(K_0C)-\tau (n-in_0)=
-\tau'(\delta)n+\tau n_0(i-\delta n)+(i-\delta n)\log(K_0C).
\end{align*}
Hence,  we can write the first summand of  \eqref{eq:expTn} as
\be\label{eq:tauc}
C_1e^{-\tau'(\delta)n}\sum_{i\le [\delta n]}e^{(i-\delta n)(\log(K_0C)+\tau n_0)}\le C_1e^{-\tau'(\delta)n}\sum_{j\ge1} e^{-j(\log(K_0C)+\tau n_0)}.
\ee
Notice that the series in the right hand side of \eqref{eq:tauc} is convergent. Therefore, by choosing $\tau'=\min\{\tau'(\delta), \delta\log(1-\varepsilon_0)\}$ we get \eqref{exp:final}.
 \end{proof}

\section{Proof in the non-invertible setting}
\label{final}

In this section we give the proof of the Theorem in the non-invertible setting. We will prove each of the required properties (G1)-(G5) in separate subsections, thus completing the proof of the existence of a tower for the product system \(  f: M \to M  \). The tail estimates obtained above in the exponential, stretched exponential and polynomial case, then complete the proof.

\subsection*{(G1) Markov property} In Section \ref{tfprod} we carried out a the construction of the collection  \(  \bar\eta  \) of subsets of \(  \bar\Delta_{0}  \) which, by the first item of Proposition \ref{key} forms a partition of \(  \bar\Delta_{0}  \) mod 0. In particular we have an induced map \(  \bar F^{T}: \bar\Delta_{0}\to \bar\Delta_{0}  \) which, using the canonical identification between \(  \bar\Delta_{0}  \) as the base of the tower and as a subset of the ambient product manifold \(  M  \),  corresponds to an induced map \(  \bar f^{T}: \bar\Delta_{0}\to \bar\Delta_{0}  \) where \(  \bar f: M \to M  \) is the product map on \(  M  \). The construction of \(  \bar\eta  \) and the induced map implies that \(  \bar f^{T}  \) satisfies the Markov property (G1).

\subsection*{(G2) Uniform Expansion}
The uniform expansivity condition follows immediately from the fact that it holds by assumption for the individual components and that the return time \(  T  \) is a sum of return times for each of the individual components. Then
\[
\|(D\bar f^{T})^{-1}(x)\|\leq \max\{\|(Df^{T})^{-1}(x)\|,  \|(D{f'}^{T})^{-1}(x)\|\}\leq \beta.
\]

\subsection*{(G3) Bounded distortion}
Recall first of all from the definition of Gibbs-Markov-Young tower the notion of \emph{separation time} and let \(  s, s'  \) denote the separation time of points in \(  \Delta_{0}, \Delta_{0}'  \)  with respect to the partitions \(  \eta, \eta'  \) respectively.
We let $D,\beta, D', \beta'$ be the distortion constants as in \eqref{bdd} for $f$ and $f'$ respectively and
let
$$\bar{D}=\text{max}\left\{\frac{D\beta}{1-\beta}, \frac{D'\beta'}{1-\beta'}\right\},
\quad \bar{\beta}=\text{max}\{\beta, \beta'\}.
$$
Now let \(  \bar s  \) denote the separation time of points in \(  \bar\Delta_{0}  \) with respect to the product map \(  \bar f  \) and the partition \(  \bar \eta  \) constructed above. Then, using the usual identification of $\bar \Delta_0$ with the subset of manifold, to prove the bounded distortion condition (G3) it is sufficient to show that for any $\Gamma\in\bar{\eta}$ and
$\bar{x}, \bar{y}\in \Gamma$ such that
$\bar s(\bar x, \bar y)<\infty$ we have
$$\left|\log\frac{\text{det}D\bar f^T(\bar x)}{\text{det}D\bar f^T(\bar y)}
\right|\leq \bar{D}\bar{\beta}^{\bar{s}(\bar{f}^T(\bar x),
\bar{f}^T(\bar y))}.
$$
To prove this, note first that, by the property of Jacobian and absolute value we have
\begin{equation}\begin{split}\label{logbdd}
 \left|\log \frac{\text{det}D\bar{f}^T(\bar{x})}{\text{det}D\bar{f}^T(\bar{y})}\right|\leq \left|\log\frac{\text{det}f^T(x)}{\text{det}Df^T(y)}\right|+
\left|\log\frac{\text{det}D{f'}^T(x')}{\text{det}D{f'}^T(y')}\right|\\
 \end{split}
\end{equation}
for all $\bar{x}=(x, x'),$$\bar{y}=(y, y')\in\Gamma.$ Moreover, notice that since the determinants are all calculated at return times \(  T  \), we can use the identification $\bar f^T=\bar F^T$ and reformulate the above expressions in terms of the Jacobians \(  JF^{T}  \).

For the first term,  notice that the simultaneous return time $T$ can be written as a sum
of return times to $\Delta_0.$ Without loss of generality, assume  $T=R_1+...+R_k,$ for some $k.$
For any $x\in\Delta_0$, let $x_0=x$ and $x_j=F^{R_1+...+R_j}(x)$ where $j= 1, ..., k-1.$
Then $$JF^T(x)=JF^{R_1}(x)JF^{R_2}(x_1)JF^{R_3}(x_2)...JF^{R_k}(x_{k-1}).$$
 Since $F^T(\pi\Gamma)=\Delta_0$ we have $F^{R_1}(\pi\Gamma)\subset\eta(x_1),$ ...,
 $F^{R_1+...+R_{k-1}}(\pi\Gamma)\subset\eta(x_{k-1}).$
 Hence, for the points $x, y\in \pi\Gamma$ the sequences $x_j$ and $y_j$ belong to
 the same element of $\eta$ for all $j=0, ..., k-1,$ which implies
  $$s(F^{R_{j+1}}(x_j), F^{R_{j+1}}(y_j))=s(F^T(x), F^T(y))+R_{j+2}+...+R_k.$$
Hence,
 \begin{equation*}\begin{aligned}
 &\left|\log\frac{JF^T(x)}{JF^T(y)}\right|\le\sum_{j=1}^k\left|\log\frac{JF^{R_j}(x)}
 {JF^{R_j}(y)}\right|\le\sum_{j=1}^k\left|\frac{JF^{R_j}(x)}{JF^{R_j}(y)}-1\right|\\
 \le
 &
 \sum_{j=0}^k D\beta^{s(F^{R_{j+1}}(x_j), F^{R_{j+1}}(y_j))} \le D\beta^{s(F^T(x), F^T(y))}
\sum_{j=0}^k
 \beta^{R_{j+2}+...+R_k}
 \\
 \le
 &
 D\beta^{s(F^T(x), F^T(y))}\sum_{j=0}^\infty
 \beta_F^j\le D\frac{\beta}{1-\beta}\beta^{s(F^T(x), F^T(y))}.
 \end{aligned}
 \end{equation*}
 The second summand is estimated similarly and we get
 $$\left|\log\frac{JF'^T(x')}{JF'^T(y')}\right|\le D'\frac{\beta'}{1-\beta'}
{\beta'}^{s'(F'^T(x'), F'^T(y'))}.$$
Using the fact that
for any  $w,z\in\bar{\Delta}_0$ we have
\(
\bar{s}(w,z)\leq \min\{s(\pi w, \pi z), s'(\pi'w, \pi'z)\}
\)
we obtain the required bound.

\subsection*{(G4) Integrability}
Follows immediately from the tail estimates obtained above.

\subsection{Aperiodicity}
As mentioned above, conditions (G1)-(G4) imply the existence of an ergodic \(  \bar f  \)-invariant probability measure \(  \bar\mu  \). Moreover it is known by standard results that this measure is mixing if and only if the aperiodicity conditions is satisfied. Thus it is sufficient to show that \(  \bar\mu  \) is mixing. To see this,  let \(  \mu, \mu'  \) be the invariant, mixing, probability measures associated to the maps \(  f, f'  \) as introduced in Section \ref{partition}. Then the measure \(  \mu \times \mu'   \) is invariant and mixing for the product map \(  \bar f  \) and thus it is sufficient to show that \(  \bar\mu = \mu \times \mu'  \) to imply that \(  \bar\mu  \) is mixing. This follows again by standard uniqueness arguments. Indeed, both \(  \bar \mu  \) and \(  \mu \times \mu'  \) are ergodic and equivalent to the reference measure, at least on the set \(  \bar\Delta_{0}  \). Thus, by Birkhoff's ergodic Theorem, for any integrable function, their time averages converge to the same limit and so \(  \int \varphi d\bar\mu = \int \varphi d\mu\times \mu'  \) implying that \(  \bar\mu=\mu \time\mu'  \).

\section{Proof in the invertible setting}\label{pofsuytower}

 Here we give proof of the Theorem in the invertible setting.
We start with the definition of Young towers in this setting, following
\cite{AlvPin}  which generalizes the definition of \cite{Y} (see
Remarks 2.3-2.5 from \cite{AlvPin}).  We change the notation slightly to distinguish this case from non-invertible case.
Let $ M$ be a Riemannian manifold.   If $\gamma\subset M$ is
a submanifold, then $m_\gamma$ denotes the restriction of the
Riemannian volume to $\gamma.$
Consider $f:M\to M$ such that  $f:M\setminus \mathcal C\to M\setminus f(\mathcal C)$ is a $C^{1+\varepsilon}$
diffeomorphism  on each connected component  of $M\setminus \mathcal C$ for some $\mathcal C\subset M$ with the following properties.

\begin{itemize}
\item[(A1)] There exists $\Lambda\subset M$ with hyperbolic product
structure, i.e. there are families of stable and unstable manifolds
$\Gamma^s=\{\gamma^s\}$ and $\Gamma^u=\{\gamma^u\}$ such that
$\Lambda=(\cup\gamma^s)\cap(\cup\gamma^u);$
dim$\gamma^s+\text{dim}\gamma^u=\text{dim}M;$ each $\gamma^s$ meets
each $\gamma^u$ at a unique point; stable and unstable manifolds are
transversal with angles bounded away from $0;$
$m_{\gamma^u}(\gamma^u\cap \Lambda)>0$ for any $\gamma^u.$
\end{itemize}
Let $\Gamma^s$ and $\Gamma^u$ be the defining families of $\Lambda.$ A
subset $\Lambda_0\subset\Lambda$ is called $s$-subset if $\Lambda_0$
also has a hyperbolic structure and its defining families can be
chosen as $\Gamma^u$ and $\Gamma^s_0\subset\Gamma^s.$ Similarly, we
define $u$-subsets. For $x\in\Lambda$ let $\gamma^\theta(x)$ denote
the element of $\Gamma^\theta$ containing $x,$ where $\theta=u, s.$

\begin{itemize}
\item[(A2)] There are pairwise disjoint $s$-subsets $\Lambda_1, \Lambda_2,
..., \subset\Lambda$ such that $m_{\gamma^u}((\Lambda\setminus\cup
\Lambda_i)\cap\gamma^u)=0$ on each $\gamma^u$ and for each
$\Lambda_i,$ $i\in\mathbb{N}$ there is $R_i$ such that
$f^{R_i}(\Lambda_i)$ is $u$-subset;
$f^{R_i}(\gamma^s(x))\subset\gamma^s(f^{R_i}(x))$ and
$f^{R_i}(\gamma^u(x))\supset\gamma^u(f^{R_i}(x))$ for any
$x\in\Lambda_i.$
\end{itemize}

\begin{itemize}
\item[(A3)] There exist constants $C\ge 1$ and $\beta\in(0, 1)$ such that
 $\text{dist}(f^n(x), f^n(y))\le C\beta^n,$ for all
$y\in\gamma^s(x)$ and $n\ge 0;$

\item[(A4)] Regularity of the stable foliation:
given $\gamma, \gamma'\in \Gamma^u$ define
$\Theta:\gamma'\cap \Lambda\to \gamma\cap \Lambda$ by
$\Theta(x)=\gamma^s(x)\cap\gamma.$ Then
\begin{itemize}
\item[(a)] $\Theta$ is absolutely continuous and
$$
u(x):= \frac{d(\Theta_\ast m_{\gamma'})}{dm_\gamma}(x)=\prod_{i=0}^{\infty}
\frac{\det Df^u(f^i(x))}{\det Df^u(f^i(\Theta^{-1}(x)))};$$
\item[(b)] There exists $C>0$ and $\beta<1$
such that, letting the separation time $s(x, y)$ be the smallest $k$
where $(f^R)^k(x)$ and $(f^R)^k(y)$ lie in different partition
elements, we have
$$\log\frac{u(x)}{u(y)}\le C\beta^{s(x, y)}\quad \text{for}\quad x, y\in\gamma'\cap \Lambda.$$
\end{itemize}

\item[(A5)] Bounded distortion: for $\gamma\in\Gamma^u$ and $x, y\in \Lambda\cap\gamma$
$$\log\frac{\det D(f^R)^u(x)}{\det D(f^R)^u(y)}\le C\beta^{s(f^R(x), f^R(y))}.$$

\item[(A6)]  $\int R\,dm_0<\infty,$ where $m_0$ is the restriction of Lebesgue measure to $\Lambda.$
\item[(A7)]  gcd$\{R_i\}=1.$
\end{itemize}

Given such a structure we can define Young-tower as we define $\Delta$ in the non-invertible case. This time we denote the tower by $\mathcal T.$
Let $F_i:\mathcal T_i\circlearrowleft$ be the two towers corresponding to maps
 $f_i: M_i \circlearrowleft,$ $i=1,
2$ as in the statement of the Theorem. Then, from conditions
(A1)-(A7) we can obtain GMY-towers by considering the system
obtained by the equivalence relation $\sim$ on $\Lambda^i,$ $i=1,2$ defined
as $x\sim y$ if and only if $y\in\gamma^s(x).$ Then on
$\Delta_0^i=\Lambda^i/\sim$ we have the partition $\mathcal
P^i=\{\Delta_{0,j}\}:=\{\Lambda_{0, j}/\sim\}$ and the return time
function $R^i:\Delta_{0, j} \to \mathbb Z^+$, and the quadruples
$(F_i, R^i, \mathcal P_i, s_i),$ $i=1,2$  satisfy conditions
(G1)-(G5). Moreover there is natural projection $\bar\pi_i:\mathcal
T\to \Delta$ that sends each stable manifold to a point.  We can then
define the direct product of these two ``quotient'' GMY-towers and,
from previous construction we obtain a new GMY tower for this
product. Thus, on $\Delta_0^1\times \Delta_0^2$ we have a partition
$\hat{\mathcal P}$, and return time $T:\Delta_0\to\mathbb N$ such
that for any $A\in\hat{\mathcal P}$ we have $(F_1\times
F_2)^T(A)=\Delta_0^1\times\Delta_0^2.$ On the other hand we know
that each $A\in\hat{\mathcal P}$ is of form $A_1\times A_2$ and
$\bar\pi_i^{-1}(A_i)\subset\Lambda^i,$ $i=1,2.$ Then $\mathcal
Q=\{\bar\pi^{-1}_1(A_1)\times\bar\pi^{-1}(A_2)| A_i\in\hat{\mathcal
P}_i, i=1,2\}$ gives the desired partition of
$\Lambda^1\times\Lambda^2.$ Indeed,   $(f_1\times
f_2)(\bar\pi^{-1}_1(A_1)\times\bar\pi^{-1}(A_2))$ is a \(  u  \)-subset of
$\Lambda^1\times\Lambda^2$
because at return times we have $f_i^T=F_i^T,$ $i=1, 2.$

All that is left is to check the properties (A1)-(A7). Those that refer to the combinatorial structure follow immediately from the discussion above, others follow immediately from the corresponding properties of the quotient tower. The only new property to check here is the second item in (A3).
This follows easily by noticing that from the definition of $T$ we have $s_T(x, y)\le \min \{s_{R^1}(x_1,
y_1), s_{R^2}(x_2, y_2)\},$ where $s_T, s_{R^1}, s_{R^2}$ denote
separation times with respect to return times $T, R^1, R^2$ and
$x=(x_1, x_2),$ $y=(y_1, y_2)$ and therefore,
by the definition of product metric, we have
\begin{align*}\text{dist}_M((f^T)^n(x),
(f^T)^n(y))=\max_{i=1,2}\{\text{dist}_i((f_i)^T(x_i),
(f_i)^T(y_i))\}\\\le C\beta^{\min\{s_{R^1}(x_1, y_1), s_{R^2}(x_2,
y_2)\}}\le C\beta^{s_T(x, y)}. \end{align*}

Tail estimates transfer directly, since we need to estimate
$m_\gamma \{T>n\}$ for $\gamma\in \Gamma^u.$

\appendix
\section{Tower estimates}\label{appendixA}

We collect here a few simple   estimates which hold in general for any Young tower, and which are used mainly in Section \ref{asymp}.
 For simplicity we state them all for the  Tower map \(  F  \), the same estimates hold also for \(  F'  \).
Recall the definition of the partitions  $\eta_n$  in \eqref{join} and for any \(  n\geq 1  \), let
$$
\eta_n^0=\{A\in\eta_n|\,\, F^n(A)=\Delta_0\}.
$$
\begin{lemma}\label{bddn}
For any $A\in\eta_n^0$ and $x, y \in A$ the following inequality holds

  \begin{equation*}
 \left|\frac{JF^n(x)}{JF^n(y)}-1\right|\leq D,
    \end{equation*}
 where $D$ as in   \eqref{bdd}.
\end{lemma}

 \begin{proof}
 The collection $\eta_n^0$ is a partition of $F^{-n}\Delta_0$ and for any $x\in\Delta_0$
each $A\in\eta_n^0$ contains a single element  of $\{F^{-n}x\}.$
 For $x\in A$ let $j(x)$ be the number of visits of its orbit to $\Delta_0$ up to time \(  n  \). Since the images of $A$
 before time $n$ will remain in an element of $\eta,$ all the points in $A$ have the same combinatorics up to time \(  n  \)  and so  $j(x)$ is constant on $A.$
Therefore $JF^n(x)=(JF^R)^j(\tilde{x}),$ for the projection $\tilde{x}$ of $x$ into
 $\Delta_0$ (i.e. if $x=(z, \ell)$ then $\tilde{x}=(z,0)$).
Thus for any $x,y\in\Delta_0$ from \eqref{bdd} we obtain
  \begin{equation}
 \left|\frac{JF^n(x)}{JF^n(y)}-1\right|=
 \left|\frac{(JF^R)^j(\tilde{x})}{(JF^R)^j(\tilde{y})}-1\right|\leq
 D.
    \end{equation}
 \end{proof}
\begin{corollary}\label{connec}
For any $A\in\eta_n^0$ and $y\in A$ we have
 \begin{equation}
JF^n(y)\geq \frac{m(\Delta_0)}{m(A)(1+D)}.
\end{equation}
\end{corollary}

\begin{proof}
Lemma \ref{bddn} implies $JF^n(x)\leq (1+D)JF^n(y).$ Integrating both sides of this inequality with
respect to $x$ over  $A$  gives
$$
m(\Delta_0)=\int_AJF^n(x)dm\leq JF^n(y)(1+D)m(A).
$$
 Hence, for any $y\in A$ we  have the statement.
\end{proof}

\begin{lemma}\label{m0}
 There exists $M_0\ge 1$ such that for any $n\in\mathbb{N}$
 $$\frac{dF^n_\ast m}{dm}\leq M_0.$$
 \end{lemma}

\begin{proof}
   Let $\nu_n=F^n_\ast m.$ We will estimate the density \(  {dF^n_\ast m}/{dm}  \) at different point \(  x\in \Delta  \) and consider three different cases according to the position of \(  x  \). First of all, for any $x\in\Delta_0$,  from Corollary \ref{connec} we have

   \begin{align*}
  \frac{d\nu_n}{dm}(x)=\sum_{y\in F^{-n}x}\frac{1}{JF^ny}&\leq
  (D+1)\sum_{A\in\eta_0^{(n)}}\frac{m(A)}{m(\Delta_0)}\\ &\leq
  (D+1)\frac{m(\Delta)}{m(\Delta_0)}:=M_0.
 \end{align*}
This proves the case $x\in\Delta_0.$
 For $x\in\Delta_\ell$ with $\ell\ge n$ we have $F^{-n}(x)=y\in\Delta_{\ell-n}.$
 Since $JF(y)=1$ for any $y\in\Delta\setminus\Delta_0,$
  $$\frac{d\nu_n}{dm}(x)=\frac{1}{JF^n(y)}=1.$$
Finally,
let $x\in\Delta_\ell,$  $\ell<n.$ Then for any $y\in F^{-n}x$ the
equality $F^{n-\ell}y=F^{-\ell}x\in\Delta_0$ holds. Hence,
$JF(F^{j}y)=1$ for all $j=n-\ell ..., n-1.$ Therefore by the chain
rule we obtain $JF^n(y)=JF^{n-\ell}(y).$ We can reduce the
problem to the first case observing
 $$\frac{d\nu_n}{dm}(x)=\sum_{y\in F^{-n}x}\frac{1}{JF^{n}y}=
 \sum_{y\in F^{\ell-n}(F^{-\ell}x)}\frac{1}{JF^{n-\ell}y}=
 \frac{d\nu_{n-\ell}}{dm}(F^{-\ell}(x)).$$
This finishes the proof.
\end{proof}

\begin{lemma}\label{mud}
 For $n>0,$ let $A\in\eta_n^0$ and $\nu=F^n_\ast(m|A).$  Then
 \begin{equation*}
  \left|\frac{\frac{d\nu(x)}{dm}}{\frac{d\nu(y)}{dm}}-1\right|\leq D,
 \end{equation*}
 for any $x,y\in\Delta_0.$
 \end{lemma}
\begin{proof}
By the assumption, $F^n:A\to\Delta_0$ is invertible.
So for any $x\in\Delta_0$ there is a unique $x_0\in A$ such that $F^n(x_0)=x$ and
$\frac{d\nu}{dm}(x)=\frac{1}{m(A)}\frac{1}{JF^nx_0}.$ Let $\varphi=\frac{d\nu}{dm}$
then for $x,y\in\Delta_0$, using  Lemma \ref{bddn} we obtain
$$\left|\frac{\varphi(x)}{\varphi(y)}-1\right|=\left|\frac{JF^n(y_0)}{JF^n(x_0)}-1\right|\leq D.$$
\end{proof}
The proof of the following corollary is analogous to the proof of Corollary  \ref{connec}.
\begin{corollary}\label{dist}
 For $n>0,$ let $A\in\eta_n^0$ and $\nu=F^n_\ast(m|A).$  Then
$$\frac{1}{(1+D)m(\Delta_0)}\le \frac{d\nu}{dm}(x)\leq \frac{1+D}{m(\Delta_0)}.$$
\end{corollary}

\section{Tail estimates for stretched exponential decay}
Here we prove Lemma \ref{basic}. First we prove the following

\begin{lemma}\label{st:tail}
Let $\tau>0,$ $\theta\in(0, 1).$ For all $n\ge (\frac{2}{\tau\theta})^{1/\theta}$ we have $$\sum_{k\ge n}e^{-\tau k^\theta}\le \frac {2}{\tau\theta}e^{-\tau n^\theta}n^{1-\theta}.$$
\end{lemma}
\begin{proof}
First of all note that  we have $$
\sum_{k\ge n}e^{-\tau k^{\theta}}\le
\int_n^{\infty}e^{-\tau x^{\theta}}dx.
$$

After change of variables $t=\tau x^\theta$ we obtain
\be\label{integ}
\int_n^{\infty}e^{-\tau x^{\theta}}dx=
\frac{1}{\theta\tau^{1/\theta}}
\int_{\tau n^{\theta}}^{\infty}e^{-t}t^{1/\theta-1} dt.
\ee
In \cite{NatPal} it was proved that for any $a, B>0$ and $x>\frac{B(a-1)}{B-1}$
\be\label{gamma}
\int_x^{\infty}t^{a-1}e^{-t}dt<Bx^{a-1}e^{-x}.
\ee
Substituting  \eqref{gamma} with $a=1/\theta$ and $B=2$
into the right hand side of \eqref{integ} finishes the proof.
\end{proof}

Now, we are ready to prove Lemma \ref{basic}. By Lemma \ref{st:tail}
and  definition of $\bm_{n}$, for sufficiently large \(  n  \)   we have
$$\bm_{k_j-n_0}\le
C'\sum_{\kappa\ge k_j-n_0} e^{-\tau\kappa^{\theta}}\le
\frac {2C'}{\tau\theta} e^{-\tau(k_j-n_0)^\theta}(k_j-n_0)^{1-\theta}$$
for any  $0\le j\le i-1$ and
$$\bm_{n-n_0-\sum k_j}\le
\frac {2C'}{\tau\theta}e^{-\tau(n-n_0-\sum k_j)^\theta}
(n-\sum k_j-n_0)^{1-\theta}.$$
Using  $a^\alpha+b^\alpha\ge (a+b)^{\alpha}$ for $\alpha\in(0, 1)$ and $a,b \ge 0$  we obtain
\begin{align*}
\bm_{n-n_0-\sum k_j}\prod_{j=1}^{i-1}\bm_{k_j-n_0} &\le \left(\frac{2C'}{\tau\theta}\right)^ie^{\tau(n-in_0)^\theta}(n-\sum_j k_j-n_0)^{1-\theta}\prod(k_j-n_0)^{1-\theta} \\
&\le \left(\frac{2C'}{\tau\theta}\right)^ie^{\tau(n-in_0)^\theta}n^{i(1-\theta)}.
\end{align*}
Taking \(  C_{1}= 2C'/\tau\theta  \) we obtain the statement in the Lemma.

\section{Decay of correlations for product measures}
Our main result can be used in conjunction with \cite[Theorem A]{AFLV} to show that rates of Decay of Correlations for a product system follow from rates of Decay of Correlations for the individual components, even without assuming a priori the existence of a Young tower, at least in the polynomial and stretched exponential cases. However, there also exists a much more general, elementary and direct proof of this statement which was told to us by C. Liverani.
 The calculation is not completely trivial and is clearly of some interest and so, since it does not appear to  be in the literature,   we reproduce it here with his kind permission.

\begin{definition}
Let $\mathcal B_1,$ $\mathcal B_2$ be Banach spaces of measurable
observables defined on $M$. We denote the \emph{correlation} of non-zero
observables $\varphi\in\mathcal B_1$ and $\psi\in \mathcal B_2$ with
respect to $\mu$ by
$$\text{Cor}_\mu(\varphi, \psi\circ
f^n):=\frac{1}{\|\varphi\|_{\mathcal B_1}\|\psi\|_{\mathcal B_2}}
  \left|\int\varphi(\psi\circ f^n)d\mu-\int\varphi d\mu\int\psi d\mu\right|.$$
We say that $(f,\mu)$ has decay of correlations at rate
$\{\gamma_n\}$  with respect to $\mu$ for observables in $\mathcal
B_1$ against  observables in \( \mathcal B_2 \) if there exists
constant $C>0$ such that for any $\varphi\in\mathcal B_1$,
$\psi\in\mathcal B_2$ the inequality
$$\text{Cor}_\mu(\varphi, \psi\circ f^n)\le C \gamma_n
$$
holds for all \( n\in \mathbb{N}\).
\end{definition}

Rates of decay of correlations have been extensively studied and are
well known for many classes of systems and various families of
observables, indeed the literature is much too vast to give complete
citations. We just mention that the first results on rates of decay
of correlations go back at least to \cite{B, LY, R, S1} in the 70's
and since then results have been obtained in \cite{AlvSch,  BG, B,
BLS, BM, ChM, Gou, HK, KN, Hu,  Li, Li2, LuzMel, NvS, PY, Th, Y1,Y, LSY}
amongst others.

Let $f_i:(M_i, \mu_i)\circlearrowleft,$  $i=1, ..., p,$  be a family
of maps defined on  compact metric spaces $M_i,$  preserving Borel
probability measures $\mu_i$ and let $f:=f_1\times \cdots\times
f_p:M\to M$ be the direct product on $M:=M_1\times \cdots\times
M_p,$ and  $\mu:=\mu_1\times \cdots\times \mu_p$ be the product
measure.  We let \( \mathcal A_i \), \( \mathcal B_i \)  denote
Banach spaces of functions on \( M_{i} \)  and \(\mathcal A\),
\(\mathcal B\)  denote Banach spaces of functions on \( M \). For
completeness we state the minimal requirements on the spaces \(
\mathcal A_{i}, \mathcal B_{i}, \mathcal A, \mathcal B  \) needed
for the calculations to work. We assume these spaces  satisfy the
following properties:

\begin{enumerate}
\item For all $1\le i\le p$  $\mathcal A_i, \mathcal B_i\subset L^2(M_i, \mu_i)$
and  $\mathcal A, \mathcal B\subset L^2(M_, \mu).$

\item For any $\varphi\in\mathcal A,$ $\psi\in\mathcal B$
 for all $1\le i\le p$ and $\mu_1\times\cdots\times\mu_{i-1}\times\mu_{i+1}\times\cdots\mu_p$
-almost every $(x_1, ..., x_{i-1}, x_{i+1}, ..., x_p)$ we have $\hat
\varphi_i:=\varphi(x_1, ..., x_{i-1}, \cdot, x_{i+1}, ..., x_p)\in
\mathcal A_i,$ and  $\hat\psi_i:=\psi(x_1, ..., x_{i-1}, \cdot, x_{i+1}, ...,
x_p)\in \mathcal B_i.$

\item There exists $C>0$ such that for any \(  \varphi\in \mathcal A, \psi\in\mathcal B\)
 and any $1\le i\le p,$  we have
$$\|\hat \varphi_i\|_{\mathcal A_i}\le C\|\varphi\|_{\mathcal
A}\quad \text{and} \quad \|\hat\psi_i\|_{\mathcal B_i}\le
C\|\psi\|_{\mathcal B}.$$
\end{enumerate}

It is easy to check that all these conditions are satisfied for some of the commonly considered classes of observables such as H\"older continuous or essentially bounded functions.
%
%\begin{theorem}\label{decor}
%Suppose $\text{Cor}_\mu(\varphi_i, \psi_i\circ f_i^n)\le C \gamma_n$
%for all non-zero \(\varphi_i\in \mathcal A_i\) and \(\psi_{i}\in
%\mathcal B_{i}\)  for all $i=1, ..., p.$ Then  there exists a
%constant $\bar C>0$   such that for all non-zero $\varphi\in
%\mathcal A$,$\psi\in \mathcal B$, and for all $n\ge 0$ we have
%$$\text{Cor}_\mu(\varphi, \psi\circ f^n)\le \bar C \gamma_n.$$
%\end{theorem}
%
We now
suppose that
\[
\text{Cor}_\mu(\varphi_i, \psi_i\circ f_i^n)\le C \gamma_n
\]
for all non-zero \(\varphi_i\in \mathcal A_i\) and \(\psi_{i}\in \mathcal B_{i}\)  for all $i=1, ..., p$ and obtain a bound for the correlation function of the product.
For simplicity we give the calculation for \(  p=2  \), the general case follows by successive applications of the argument.
 Let $\varphi\in
\mathcal A$ be such that $\int \varphi d\mu=0$ and $\psi\in \mathcal
B.$ Moreover, let $\bar\varphi(x_1)=\int\varphi(x_1, y)d\mu_2(y).$
If we fix the first coordinate, by Fubini's theorem we have
\begin{equation}
\begin{aligned}\label{corr}
&\int \varphi(x_1, x_2)\psi(f^n_1(x_1),  f^n_2(x_2)) d\mu_1d\mu_2 \\
=
&\int\left(\int\psi(f^n_1(x_1), f^n_2(x_2))[\varphi(x_1,
x_2)-\bar\varphi(x_1)]d\mu_2\right)d\mu_1 \\ +&
\int\bar\varphi(x_1)\psi(f^n_1(x_1), f^n_2(x_2))d\mu_1d\mu_2.
\end{aligned}
\end{equation}
Since  $\mu_2$ is $f_2$-invariant, we can write
the first term of the right hand side as
\begin{align*}
I:&=\int\psi(f^n_1(x_1), f^n_2(x_2))\varphi(x_1,
x_2)d\mu_2-\bar\varphi(x_1)\int\psi(f^n_1(x_1), f^n_1(x_2))d\mu_2
\\ &\le
C\|\varphi(x_1, \cdot)\|_{\mathcal A_2}\|\psi(x_1,
\cdot)\|_{\mathcal B_2} \gamma_{n}.
\end{align*}
The inequality above follows since the left hand side gives the
correlations of the second component with respect to $\mu_2.$ From
the third assumption we obtain $I\le C\|\varphi\|_{\mathcal
A}\|\psi\|_{\mathcal B}$.
Again by the invariance of $\mu_2$ we can write the second summand
of the equation \eqref{corr} as
\begin{gather*}
\int\psi(f^n_1(x_1), f^n_2(x_2))\bar\varphi(x_1)d\mu_1d\mu_2=
 \int\psi(f^n_1(x_1), x_2)\bar\varphi(x_1)d\mu_1d\mu_2.
\end{gather*}
Note that $\int\bar\varphi(x_1)d\mu_1=0$ by the choice of $\varphi.$
Then this  expression can be written as

\begin{align*}
&\int\left(\int\psi(f^n_1(x_1),x_2)\bar\varphi(x_1)d\mu_1\right)d\mu_2\\
=&\int\left(\int\psi(f^n_1(x_1),x_2)[\bar\varphi(x_1)d\mu_1-\int\bar\varphi(x_1)d\mu_1]\right)d\mu_2\\
=&\int\left(\int\psi(f^n_1(x_1),x_2)\bar\varphi(x_1)d\mu_1-
\int\psi(x_1, x_2)d\mu_1\int\bar\varphi(x_1)d\mu_1\right)d\mu_2
\end{align*}
The expression under the integral with respect to $\mu_2$ is exactly
the correlation with respect to $\mu_1.$ Again using the third
property of the Banach spaces $\mathcal A$ and $\mathcal B$ we have
$$\text{Cor}_{\mu}(\varphi, \psi\circ f_1\times f_2)\le
2C\gamma_n.$$
This completes the proof of the required statement.
%\|\varphi\|_{\mathcal A}\|\psi\|_{\mathcal B}

\end{document}